\documentclass{article}
\usepackage{color}
\usepackage{hyperref}
\usepackage{url}
\usepackage{latexsym}
\usepackage{amsmath}
\usepackage{amssymb,amsfonts,amscd}
\usepackage{graphicx}
\usepackage{multirow,comment}

\newtheorem{theorem}{Theorem}
\newtheorem{proposition}[theorem]{Proposition}

\newtheorem{rem}[theorem]{Remark}
\newtheorem{defn}[theorem]{Definition}
\newtheorem{ex}[theorem]{Example}

\newenvironment{remark}{\begin{rem}\em}{\end{rem}}

\newenvironment{example}{\begin{ex}\em}{\end{ex}}
\newenvironment{proof}{{\noindent\bf Proof.\ }}{\qed}

\newcommand{\grad}{{\nabla}}
\newcommand{\Var}{\mathcal V}

\newcommand{\bN}{{\mathbb N}}
\newcommand{\bC}{{\mathbb C}}
\newcommand{\bP}{{\mathbb P}}

\newcommand{\bR}{{\mathbb R}}

\newcommand{\sC}{{\mathcal C}}
\newcommand{\sF}{{\mathcal F}}
\newcommand{\sL}{{\mathcal L}}
\newcommand{\sJ}{{\mathcal J}}

\newcommand{\sM}{{\mathcal M}}

\newcommand\qed{{\hspace*{\fill}$\Box$\vskip12pt plus 1pt}}

\newcommand{\Cone}[1]{{\mathcal C}(#1)}

\DeclareMathOperator{\Sym}{{\rm Sym}}
\DeclareMathOperator{\rk}{rk}
\DeclareMathOperator{\brk}{brk}
\DeclareMathOperator{\Null}{null}
\DeclareMathOperator{\codim}{codim}

\begin{document}

\title{Tensor decomposition and homotopy continuation}

\author{Alessandra Bernardi\thanks{Dipartimento di Matematica,
Universit\`a di Trento, via Sommarive 14, I-38123 Povo (TN) Italy
(alessandra.bernardi@unitn.it, \url{http://me.unitn.it/alessandra-bernardi}).
This author was partially supported by Institut Mittag Leffler the Royal Swedish Academy of Science (Sweden), Inria Sophia Antipolis M\'editerran\'ee (France), Dipartimento di Matematica, Universit\`a di Bologna, GNSAGA of INDAM (Italy), and the Simons Institute for the Theory of Computing (CA, USA).}
\and
Noah S. Daleo
\thanks{Department of Mathematics, Worcester State University,
Worcester, MA 01602 (ndaleo@worcester.edu, \url{www.worcester.edu/noah-daleo}).
This author was supported in part by NCSU Faculty Research and Development Fund and NSF grant DMS-1262428.}
\and
Jonathan D. Hauenstein
\thanks{Department of Applied and Computational Mathematics and Statistics,
University of Notre Dame, Notre Dame, IN 46556 (hauenstein@nd.edu, \url{www.nd.edu/\~jhauenst}).
This author was supported in part by Army YIP, Sloan Research Fellowship, and NSF grant ACI-1460032.}
\and
Bernard Mourrain
\thanks{Inria Sophia Antipolis M\'editerran\'ee, BP 93, 06902 Sophia Antipolis, France (Bernard.Mourrain@inria.fr, \url{www-sop.inria.fr/members/Bernard.Mourrain/})
}
}

\date{July 6, 2016}

\maketitle

\begin{abstract}

\noindent
A computationally challenging classical elimination theory 
problem is to compute polynomials which vanish 
on the set of tensors of a given rank.  
By moving away from computing polynomials via elimination theory
to computing pseudowitness sets via numerical elimination theory, 
we develop computational methods for computing
ranks and border ranks of tensors along with decompositions.
More generally, we present our approach using 
joins of any collection of irreducible and 
nondegenerate projective varieties $X_1,\ldots,X_k\subset\bP^N$
defined over $\bC$.
After computing ranks over $\bC$, 
we also explore computing real ranks.  
A variety of examples are included to demonstrate
the numerical algebraic geometric approaches.

\noindent {\bf Key words and phrases.} tensor rank, homotopy continuation,
numerical elimination theory, witness sets, numerical algebraic geometry, joins, secant varieties.

\noindent {\bf 2010 Mathematics Subject Classification.}
Primary 65H10; Secondary 13P05, 14Q99, 68W30.


\end{abstract}

\normalsize

\pdfbookmark[1]{Introduction}{Sec:intro}
\section*{Introduction}\label{Sec:intro}

Computing tensor decompositions is a fundamental problem in numerous application areas including computational complexity, signal processing for telecommunications \cite{Co2,deLC}, scientific data analysis \cite{JS,SBG}, electrical engineering \cite{ACCF}, and statistics~\cite{McC}.
Some other applications include the complexity of matrix multiplication \cite{Stra83:laa}, the complexity problem of
P versus NP \cite{Vali}, the study of entanglement in quantum physics~\cite{entanglement}, matchgates in computer science \cite{Vali},
the study of phylogenetic invariants \cite{ar}, independent component analysis~\cite{Comon}, blind identification in signal processing \cite{SGB}, branching structure in diffusion images~\cite{ScSe}, and other multilinear data analytic techniques in bioinformatics and spectroscopy \cite{CoJu}.

One computational algebraic geometric approach 
for deciding if a decomposition can exist is to compute
equations that define secant and join varieties (e.g.,
see \cite[Chap.~7]{JML} for a general overview).  
This can be formulated as a classical elimination theory 
question which, at least in theory, can be computed using Gr\"obner
basis methods.  Moreover, the defining equations do not yield
decompositions, only existential information.
Rather than focusing on computing defining equations, this paper uses numerical algebraic 
geometry (e.g, see \cite{BertiniBook,SW05} for a general overview)
for performing membership tests and computing decompositions.
In particular, we use {\em numerical elimination theory}
to perform the computations based on the methods developed in
\cite{HS13,HS10} (see also \cite[Chap.~16]{BertiniBook}).
This approach differs from several previous methods 
of combining numerical algebraic geometry and elimination theory, e.g., \cite[\S~3.3-3.4]{Exactness} and \cite{SVW01,SVW03}, 
in that these previous methods relied upon interpolation.

The general setup for this paper is as follows.
Let $X\subset\bP^N$ be an irreducible and 
nondegenerate projective variety defined over $\bC$
and $\Cone{X}\subset\bC^{N+1}$ be the affine cone of
$X$.  We let a point $P$ be a nonzero vector 
in $\mathbb{C}^{N+1}$ while 
$[P]$ denotes the line in $\mathbb{C}^{N+1}$ passing through 
the origin and $P$, i.e., $[P]\in \bP^N$ is the 
projectivization of $P\in\bC^{N+1}$.  
The {\it $X$-rank} of $[P]\in \mathbb{P}^N$ (or of $P\in \bC^{N+1}$),
denoted $\rk_X(P)$, is the minimum $r\in \mathbb{N}$ such that $P$ can be written as a linear combination of $r$ elements of $\Cone{X}$:
\begin{equation}\label{eq:JoinDefEq}
P=\sum_{i=1}^r x_i, \; \; \; x_i \in \Cone{X}.
\end{equation}

Let $\sigma_r^0(X)\subset\bP^N$ denote the set of elements
with rank at most $r$ and, for $[x_i]\in \bP^N$, let
$\langle [x_1], \ldots , [x_r]\rangle$ denoted the linear
space spanned by $x_1,\dots,x_r$.  The {\it $r^{\rm th}$ secant variety} of $X$ is 
$$\sigma_r(X) = \overline{\sigma_r^0(X)} = \overline{\bigcup_{[x_1], \ldots , [x_r]\in X}\langle [x_1], \ldots , [x_r]\rangle}.$$ 
In particular, if ${[P]}\in\sigma_r(X)$, then $[P]$ is
the limit of a sequence of elements of $X$-rank at most~$r$.
The {\it $X$-border rank} of ${[P]}$, denoted $\brk_X(P)$, is the minimum
$r\in\bN$ such that ${[P]}\in\sigma_r(X)$.  
Obviously, $\brk_X({P})\leq \rk_X({P})$. 

Secant varieties are just special cases of join varieties. 
For irreducible and nondegenerate projective varieties 
$X_1, \ldots, X_k$, the 
constructible join and join variety of $X_1, \ldots , X_k$, respectively, are
\begin{equation}\label{eq:JoinDef}
J^0(X_1 , \ldots , X_k)= \bigcup_{[x_1]\in X_1, \ldots , [x_k]\in X_k}\langle [x_1] ,\ldots , [x_k] \rangle
\hbox{~~and~~}
J(X_1 , \ldots , X_k)= \overline{J^0(X_1,\ldots,X_k)}.
\end{equation}
Clearly, 
$\sigma_r^0(X)=J^0(\underbrace{X,\ldots,X}_r)$
and
$\sigma_r(X)=J(\underbrace{X,\ldots,X}_r)$.

As mentioned above, one can test, in principle, 
if an element belongs to a certain join variety 
(or if it has certain $X$-border rank) 
by computing defining equations for the join variety (or the secant variety, respectively). Unfortunately, finding defining equations for secant and join varieties is generally a very difficult elimination problem which is far from being well understood at this time.

The following summarizes the remaining sections of this paper.

The knowledge of whether an element lies in a constructible
join (or if it has a certain $X$-rank) is an open condition.
In fact, $\sigma_r^0(X)$ is almost always an open subset of $\sigma_r(X)$, so membership tests for $\sigma_r^0(X)$ 
based on evaluating polynomials to determine 
the $X$-rank of a given element do not exist in general.
Currently, there are few theoretical algorithms for specific cases, {e.g., \cite{babe2,bgi,bbcm,BCMT,CS,IK,OO,Syl}.}
In Section~\ref{Sec:Membership}, we present 
a numerical algebraic geometric approach to join varieties.

Once an element is known to be in 
$J(X_1,\dots,X_k)$ or $\sigma_r(X)$,
numerical elimination theory can also be used to decide if 
the element is in the corresponding constructible set
$J^0(X_1,\dots,X_k)$ or $\sigma_r^0(X)$.  
Rather than test a particular element, 
Section~\ref{Sec:Boundary} considers
the approach first presented in \cite{HS13} 
for computing the codimension one components
of the boundaries of $J^0(X_1,\dots,X_k)$ or $\sigma_r^0(X)$,
namely the codimension one components of 
$\overline{{J(X_1,\dots,X_k)}\setminus{J^0(X_1,\dots,X_k)}}$
and 
$\overline{{\sigma_r(X)}\setminus{\sigma_r^0(X)}}$.
For example, this allows one to compute the codimension one 
component of the closure of the set of points of $X$-border rank at 
most $r$ whose rank is strictly larger than $r$.

Another problem considered in this paper from the numerical point of view regards computing real decompositions 
of a real element.  For example, after computing
the $X$-rank $r$ of a real element ${P}$, we would like to know if
it has a decomposition using $r$ real elements,
that is, determine if the real $X$-rank of $P$ is the same 
as its complex $X$-rank.   
Computing the real $X$-rank 
has recently been studied by various authors, especially
in regards to the typical ranks of symmetric tensors, 
i.e., $r$ such that the symmetric tensors whose real rank 
is $r$ is an open real set.  
Theoretic results in this direction are provided in 
\cite{Ba,Ban,bbo,Ble,BZ,CRe,CO}.
In Section~\ref{Sec:RealRank}, we describe a method using 
\cite{H13} to determine the existence of real decompositions.

We emphasize that a numerical algebraic geometric approach
works well for decomposing generic elements. 
For example, a numerical algebraic geometric based 
approach was presented in \cite{HOOS} for computing 
the total number of decompositions of a general element
in the so-called perfect cases, i.e., when the general
element has finitely many decompositions.  
In every case, one can track a single solution path
defined by a Newton homotopy to compute a decomposition
of a general element, as shown in Section~\ref{Sec:LargeRank}.

Local numerical techniques exist for computing (numerically approximating) decompositions of elements, e.g., see~\cite{KBSIAMReview,Smirnov}.  
Section~\ref{Sec:LargeRank} also develops an approach for 
combining such local numerical techniques with Newton homotopies
for deriving upper bounds on border rank
over both the real and complex numbers.

The last section, Section~\ref{Sec:Examples}, considers a variety
of examples.  

\section{Membership tests}\label{Sec:Membership}

Let $X_1,\dots,X_k\subset\bP^N$ be irreducible and nondegenerate
projective varieties.  Consider the constructible join 
$J^0(X_1,\dots,X_k)$
and join variety $J(X_1,\dots,X_k)$ defined in \eqref{eq:JoinDef}.
One key aspect of this numerical algebraic geometric approach 
is to consider the smooth irreducible variety, called
the {\em abstract join variety},
\begin{equation}\label{eq:TotalSpace}
\sJ = \left\{({[P]},x_1,\dots,x_k)~\left|~x_i\in \Cone{X_i}, P =
  \sum_{i=1}^k x_i\right\}\right.\subset \bP^N\times \Cone{X_1} \times
\cdots \times \Cone{ X_k},
\end{equation}
where {$\Cone{X_i}$ is the affine cone of $X_i$}.
For the projection $\pi({[P]},x_1,\dots,x_k) = {[P]}$, it is clear that
\begin{equation}\label{eq:JoinProjection}
\pi(\sJ) = J^0(X_1,\dots,X_k) \hbox{~~and~~} \overline{\pi(\sJ)} = J(X_1,\dots,X_k).
\end{equation}
The key to using the numerical elimination theory approaches of \cite{HS13,HS10}
is to perform all computations on the abstract join variety $\sJ$.
Moreover, one only needs a numerical algebraic geometric description,
i.e., a witness set or a pseudowitness set,
which we define in Section \ref{Ssec:WitnessSets},
of each irreducible variety $X_i$ to perform such computations on $\sJ$.

Since $\sJ$ naturally depends on affine cones, we will 
simplify the numerical algebraic geometric presentation
by just considering affine varieties.
As \cite[Remark~8]{HS10} states, we can naturally extend from affine
varieties to projective spaces by considering coordinates as sections
of the hyperplane section bundle and accounting for the fact
that coordinatewise projections have a {\em center}, i.e.,
a set of indeterminacy, that is contained in each fiber.
Another option is to restrict to a 
general affine coordinate patch and introduce scalars
as illustrated in Section~\ref{Ssec:BasicMembershipEx}.
Due to this implementation choice, there is the potential for
ambiguity in Section~\ref{Ssec:BasicMembership}, 
e.g., the dimension of the affine cone over a projective variety
is one more than the dimension of the projective variety.
In that section, the meaning of dimension is dependent 
on the implementation choice.

After defining witness sets, we explore a membership test for the join variety $J(X_1,\dots,X_k)$
in Section~\ref{Ssec:BasicMembership}.
This test uses homotopy continuation 
without the need for computing defining equations, e.g., via interpolation
or classical elimination, 
for $J(X_1,\dots,X_k)$.

\subsection{Witness and pseudowitness sets}\label{Ssec:WitnessSets}

The fundamental data structure in numerical algebraic
geometry is a witness set, with numerical elimination theory
relying on pseudowitness sets first described in \cite{HS10}.  
For simplicity, we provide a brief overview of both 
in the affine case with more details available
in \cite[Chap. 8 \& 16]{BertiniBook}.

Let $X\subset\bC^N$ be an irreducible variety.
A {\em witness set} for $X$ is a triple $\{f,\sL,W\}$
where $f\subset\bC[x_1,\dots,x_N]$ such that $X$
is an irreducible component of
\mbox{$\Var(f) = \{x\in\bC^N~|~f(x) = 0\}$,}
$\sL$~is a linear space in $\bC^N$
with $\dim \sL = \codim X$ which intersects $X$ transversally,
and $W := X \cap\sL$.  In particular, $W$ is a collection
of $\deg X$ points in $\bC^N$, called a {\em witness point set}.

If the multiplicity of $X$ with respect to $f$ is greater than $1$,
we can use, for example, isosingular deflation~\cite{Iso}
or a symbolic null space approach of \cite{HMS15},
to replace $f$ with another polynomial system $f'\subset\bC[x_1,\dots,x_N]$
such that $X$ has multiplicity $1$ with respect to $f'$.
Therefore, without loss of generality, we will 
assume that $X$ has multiplicity~$1$ with respect to 
its {\em witness system} $f$.
That is, $\dim X = \dim\Null Jf(x^*)$ for
general $x^*\in X$ where $Jf$ is the Jacobian matrix of~$f$.

\begin{example}\label{ex:WitnessSets}
For illustration, consider
the irreducible variety $X:=\Var(f)\subset\bC^3$ where
$f = \{x_1^2 - x_2, x_1^3 + x_3\}$.
The triple $\{f,\sL,W\}$ forms a witness set for $X$ where
$\sL := \Var(2x_1 - 3x_2 - 5x_3 + 1)$
and $W := X\cap\sL$ which, to 3 decimal places, is the following set of three
points:
$$\left\{\begin{array}{l} 
           (-0.299,0.089,0.027),
            (0.450\pm0.683\cdot\sqrt{-1},-0.265\pm0.614\cdot\sqrt{-1},0.539\mp0.095\cdot\sqrt{-1})
\end{array}\right\}$$
We note that $\sL$ was defined using 
small integer coefficients for presentation purposes
while, in practice, such coefficients are selected
in the complex numbers using a random number generator.
\end{example}

A witness set for $X$ can be used to decide membership in $X$ \cite{SVW01}.  Suppose that $p\in\bC^N$ and
$\sL_p\subset\bC^N$ is a linear space passing through $p$ 
which is transverse to $X$ at $p$ 
with $\dim \sL_p = \dim \sL = \codim X$.
With this setup, $p\in X$ if and only if $p\in X\cap\sL_p$
which can be decided by deforming from $X\cap\sL$.
That is, one computes the convergent endpoints, at $t = 0$, 
of the $\deg X$ paths starting, at $t = 1$, from the points in $W$
defined by the homotopy $X\cap(t\cdot \sL + (1-t)\cdot \sL_p)$.
In particular, $p\in X$ if and only if $p$ arises as an endpoint.

Suppose now that $\pi:\bC^N\rightarrow\bC^M$ is the projection 
defined by $\pi(x_1,\dots,x_N) = (x_1,\dots,x_M)$ 
and $Y := \overline{\pi(X)}\subset\bC^M$.  
Consider the matrix $B = [I_M~0]\in\bC^{M\times N}$ so that $\pi(x) = Bx$.  A {\em pseudowitness set} for $Y$~\cite{HS10}
is a quadruple $\{f,\pi,\sM,U\}$ which is built from a
witness set for $X$, namely $\{f,\sL,W\}$, as follows.
First, one computes the dimension of $Y$, for example, 
using \cite[Lemma~3]{HS10}, namely
\begin{equation}\label{eq:DimensionImage}
\dim Y = \dim X - \dim\Null\left[\begin{array}{c} Jf(x^*) \\ B \end{array}\right]
\end{equation}
for general $x^*\in X$.  

Suppose that $\sM_1\subset\bC^M$
is a general linear space with $\dim \sM_1 = \codim Y$~and \mbox{$\sM_2\subset\bC^N$} is a general linear
space with $\dim \sM_2 = \codim X - \codim Y =: \dim_{gf}(X,\pi)$,
i.e., the dimension of a general fiber of~$X$ with respect to $\pi$.
Let \mbox{$\sM := (\sM_1\times\bC^{N-M})\cap \sM_2$}.
We assume that $\sM_1$ and $\sM_2$ are chosen to be sufficiently
general so that 
$\dim \sM = \dim \sL = \codim X$
and $U := X\cap\sM$ consists 
of \mbox{$\deg Y\cdot\deg_{gf}(X,\pi)$} points
where $\deg_{gf}(X,\pi)$ is the degree of a general
fiber of~$X$ with respect to~$\pi$.
Thus, for the {\em pseudowitness point set} $U$,
$\deg Y = |\pi(U)|$ and \mbox{$\deg_{gf}(X,\pi) = |U|/|\pi(U)|$}.

\begin{remark}\label{RemJoin}
Relation \eqref{eq:DimensionImage} provides an approach for determining if the join variety is
{\em defective}, i.e., smaller than the expected dimension. 
In fact, since the abstract join $\sJ$ of (\ref{eq:TotalSpace}) always
has the expected dimension, namely $\sum_{i=1}^{k}\dim X_{i}$, 
we may take $Y$ to be the join so that $X=\sJ$ with 
$Y = \overline{\pi(\sJ)}$.
\end{remark}

\begin{example}\label{ex:PseudoWitnessSets}
Continuing with the setup from Ex.~\ref{ex:WitnessSets},
consider the map $\pi(x_1,x_2,x_3) = (x_1,x_2)$.
Clearly, $Y := \overline{\pi(X)}$ is the parabola defined by $x_2 = x_1^2$, 
but we will proceed from the witness set for $X$ to construct a pseudowitness
set for $Y$.

We have
$$B = \left[\begin{array}{ccc} 1 & 0 & 0 \\ 0 & 1 & 0 \end{array}\right]
\hbox{~~and~~}
\dim\Null\left[\begin{array}{ccc} 2x_1^* & -1 & 0 \\ 0 & 3(x_1^*)^{2} & x_3^* \\
1 & 0 & 0 \\
0 & 1 & 0 \end{array}\right] = 0
\hbox{~for general~} (x_1^*,x_2^*,x_3^*)\in X.$$
Thus, $\dim Y = \dim X = 1$ and we can take 
$\sM := \Var(2x_1 - 3x_2 + 1)\times\bC\subset\bC^3$.

We can compute the pseudowitness point set $U := X\cap\sM$
starting from the three points in $W = X\cap\sL$
using the homotopy defined by 
$X\cap(t\cdot \sL + (1-t)\cdot \sM)$.
For this homotopy, two paths converge and one diverges
where the two convergent endpoints forming $U$ are
  $$(1,1,-1),~~~~(-1/3,1/9,1/27).$$
In particular, $\deg Y = |\pi(U)| = 2$ and $\deg_{gf}(X,\pi) = |U|/|\pi(U)| = 1$, 
meaning that $\pi$ is generically a one-to-one map from $X$~to~$Y$.
\end{example}

\begin{example}\label{ex:Curves}

As an example of computing a 
pseudowitness set for a join of varieties 
which are not rational, we consider curves 
$C_i\in\bP^4$ which are defined by
the intersection of $3$ random quadric 
hypersurfaces.  Hence, 
$C_i = \Var(f_{i1},f_{i2},f_{i3})$
where $f_{ij}$ has degree $2$
so that each $C_i$ is a complete intersection 
with $\deg C_i = 2^3 = 8$.
Consider the abstract joins 
$$
\begin{array}{rcl}
\sJ_{12} &=& \{([P],x_1,x_2)~|~P = x_1 + x_2, f_{ij}(x_i) = 0 \hbox{~for~} i = 1,2 \hbox{~and~} j = 1,2,3\}, \\
\sJ_{11} &=& \{([P],x_1,x_2)~|~P = x_1 + x_2, f_{1j}(x_i) = 0 \hbox{~for~} i = 1,2 \hbox{~and~} j = 1,2,3\}.
\end{array}
$$
That is, for $\pi([P],x_1,x_2) = [P]$, 
we have
$\overline{\pi(\sJ_{12})} = J(C_1,C_2)$
and 
$\overline{\pi(\sJ_{11})} = J(C_1,C_1) = \sigma_2(C_1)$.

Witness sets for $\sJ_{12}$ and $\sJ_{11}$
show that both abstract join varieties have degree $64$.  
Then, by converting from a witness
set to a pseudowitness set as described above, we
find that $J(C_1,C_2)$ is a hypersurface
of degree $64$ while $\sigma_2(C_1)$ is a
hypersurface of degree $16$ with $\deg_{gf}(\sJ_{11},\pi) = 2$.
\end{example}

\begin{example}\label{ex:Claw}

A key component of the proof of \cite[Thm.~1]{LSClaw}
is computing the dimension of 
$$Y = \overline{\left\{x_{h i}^{11} x_{j k}^{12} x_{\ell m}^{13} + 
x_{h i}^{21} x_{j k}^{22} x_{\ell m}^{23}~\left|~x_{\alpha\beta}^{\gamma\delta}\in\bC,~h,i,j,k,\ell,m\in\{0,1\},~i+k+m\equiv 0 \mbox{~mod~} 2 \right\}\right.} \subset\bC^{32}$$
We can compute $\dim Y$ using \eqref{eq:DimensionImage} as follows.
Let $X = \Var(f)\subset\bC^{32}\times\bC^{24}$ where
$$f(z,x) = \left[\begin{array}{c}
x_{h i}^{11} x_{j k}^{12} x_{\ell m}^{13} + 
x_{h i}^{21} x_{j k}^{22} x_{\ell m}^{23}
- z_{h i j k \ell m} \\ h,i,j,k,\ell,m\in\{0,1\},~i+k+m\equiv 0 \mbox{~mod~} 2 \end{array}\right]$$
and $\pi(z,x) = z$ so that 
$Y = \overline{\pi(X)}$ and $B = [I_{32}~0]$.  
Thus, for general $x^*\in\bC^{24}$ with $z^*\in\bC^{32}$
such that $f(z^*,x^*) = 0$, we have
$$\dim Y = \dim X - \dim\Null\left[\begin{array}{c} Jf(z^*,x^*) \\ B\end{array}\right] = 24 - \dim\Null J_xf(x^*) = 24 - 4 = 20$$
where $J_xf(x)$ is the Jacobian matrix of $f$ with respect to $x$ which only depends on $x$.
In fact, this is the largest dimension possible since
$Y$ arises as a projection of $\sigma_2(\bC^4\times\bC^4\times\bC^4)$
with $\dim \sigma_2(\bC^4\times\bC^4\times\bC^4) = 20$.
\end{example}

In practice, we may compute a pseudowitness point set $U$ 
by starting with one sufficiently general point in the image and performing monodromy loops. 
Such an approach has been used in various applications, e.g., \cite{DHO14,MR15}, 
and will be used in many of the examples in Section \ref{Sec:Examples}.
Since $Y = \overline{\pi(X)}$, we can compute a general point $y\in Y$
given a general point $x\in X$.  Then, pick a general linear space $L$ passing 
through $y$ so that $y\in U=Y\cap L$.
A random monodromy loop consists of two steps, each of which is performed using a homotopy. 
First, we move $L$ to some other general linear space $L'$. 
Next, we move back to $L$ via a randomly chosen path.
During this loop, the path starting at $y\in U$ may arrive at some other point in $U$. 
We repeat this process until no new points are found for several loops.
The completeness of the set is verified via a \textit{trace test}. 
More information about this procedure can be found in, e.g., \cite{DHO14,MR15} and \cite[\S~2.4.2]{D15}. 

The following discusses extending the membership test from witness sets to pseudowitness sets.

\subsection{Membership test for images}\label{Ssec:BasicMembership}

As mentioned above, we can extend the notion of pseudowitness sets
from the affine to the projective case.  For the join variety
$J := J(X_1,\dots,X_k) = \overline{\pi(\sJ)}$ where $\sJ$ is the abstract
join variety, we will simply assume that 
we have a pseudowitness set $\{f,\pi,\sM,U\}$ for $J$.
This pseudowitness set for $J$ provides the required information needed to 
decide membership in the join variety \cite{HS13}.
As with the membership test using a witness set,
testing membership in $J$ requires the tracking of
at most $\deg J$ many paths, i.e., one only needs
$U'\subset U$ with $\pi(U') = \pi(U)$ as discussed in \cite[Remark~2]{HS13}.

Let $d:=\dim J $ 
and suppose that 
$\sM_1$ is the corresponding sufficiently general 
codimension~$d$ linear space from the pseudowitness set
with $\pi(U) = \pi(U') = J\cap\sM_1$.

Given a point {$[P]\in J$}, suppose that
$\sL_P$ is a sufficiently general 
linear space of codimension~$d$ passing through $[P]$.
As with the membership test using a witness set,
we want to compute $J\cap\sL_p$ from $J\cap\sM_1$.
That is, we consider the $\deg J$ paths 
starting, at $t = 1$, from the points in $\pi(U') = \pi(U)$
defined by $J\cap(t\cdot \sM_1 + (1-t)\cdot \sL_P)$.
Since polynomials vanishing on $J$ are not accessible, we
use the pseudowitness set for $J$ to lift these paths 
to the abstract join variety~$\sJ$ which, by assumption, is an irreducible 
component of~$\Var(f)$. 
Thus, $f$ permits path tracking on the
abstract join variety $\sJ$ and hence permits the tracking along the join variety $J$.  
Given $w\in U'$, let $Z_w(t)$ denote the path defined on $\sJ$
where $Z_w(1) = w$.  In particular, we only need to consider $U'\subset U$ since, for any $w'\in U$ with $\pi(w) = \pi(w')$, 
$\pi(Z_w(t)) = \pi(Z_{w'}(t))$. 
With this setup we have the following membership test, which is an expanded version of \cite[Lemma~1]{HS13}.

\begin{proposition}\label{prop:BasicMembership}
For the setup described above with $J^0 := J^0(X_1,\dots,X_k)$, the following hold.
\begin{enumerate}
\item\label{Item1} $[P]\in J$ if and only if there exists $w\in U'$ such that $\lim_{t\rightarrow0} \pi(Z_w(t)) = [P]$.
Moreover, the multiplicity of $[P]$ with respect
to $J$ is equal to $\left|\{w\in U'~|~\lim_{t\rightarrow0} \pi(Z_w(t)) = [P]\}\right|$.
\item\label{Item2} $[P]\in J^0$ if there exists $w\in U$ such that \mbox{$\lim_{t\rightarrow0} \pi(Z_w(t)) = [P]$} and  $\lim_{t\rightarrow0} Z_w(t)\in \sJ$.
\item\label{Item3} If, for every $w\in U'$, $\lim_{t\rightarrow0} Z_w(t)\in \sJ$, then $[P]\in J^0$ if and only if 
there exists $w\in U'$ such that $\lim_{t\rightarrow0} \pi(Z_w(t)) = [P]$.
\item\label{Item4} Let $E_P = \{w\in U~|~[P]=\lim_{t\rightarrow0} \pi(Z_w(t))\}$.  If $E_P\neq\emptyset$ and
$\lim_{t\rightarrow0} Z_w(t)$ does not exist in $\sJ$ for every $w\in E_P$, then either $[P]\in \sJ\setminus\sJ^0$ or
$\dim \left(\pi^{-1}([P])\cap\sJ\right) > \dim_{gf}(\sJ,\pi)$.
\item\label{Item5} If $\dim J = 1$, then $[P]\in J^0$ if and only if 
there exists $w\in U$ such that \mbox{$\lim_{t\rightarrow0} Z_w(t)\in\sJ$}
with $\lim_{t\rightarrow0}\pi(Z_w(t)) = [P]$.
\end{enumerate}
\end{proposition}
\begin{proof} 
With the setup above, we know that $J\cap\sL_P$ consists
of finitely many points.  Thus, it follows from~\cite{CoeffParam}
that $[Q]\in J\cap\sL_P$ if and only if there exists $w\in U'$ such that
$[Q] = \lim_{t\rightarrow0}\pi(Z_w(t))$.
Item~\ref{Item1} follows since $[P]\in J$ if and only if $[P]\in J\cap\sL_P$ with the number of such distinct paths ending at $[P]$ 
being the multiplicity of $[P]$ with respect to $J$.

Item~\ref{Item2} follows from $\pi(\sJ) = J^0$.
In fact, $L = \lim_{t\rightarrow0} Z_w(t)\in\sJ$ with $[P]= \pi(L)$.

The assumption in Item~\ref{Item3} yields $J\cap\sL_P = J^0\cap\sL_P$.  
Thus, this item follows immediately from Item~\ref{Item1} since
$[P]\in J^0$ if and only $[P]\in J^0\cap\sL_P = J\cap\sL_P$.

For Item~\ref{Item4}, if $\dim \left(\pi^{-1}([P])\cap\sJ\right) = \dim_{gf}(\sJ,\pi)$, then 
it follows from \cite{CoeffParam} that 
there must exist $w\in U$ such that $\lim_{t\rightarrow0} Z_w(t)\in \sJ$
with $[P]=\lim_{t\rightarrow0} \pi(Z_w(t))$.  The statement follows
since $\dim \left(\pi^{-1}([P])\cap\sJ\right) < \dim_{gf}(\sJ,\pi)$
implies $\pi^{-1}([P])\cap\sJ = \emptyset$, i.e., $[P]\notin\sJ^0$.

For Item~\ref{Item5}, since $\sJ$ is irreducible with $\dim J = 1$, 
we know $\dim_{gf}(\sJ,\pi) = \dim \sJ - 1$.
Hence, every fiber must be either empty or have 
dimension equal to $\dim_{gf}(\sJ,\pi)$.  
\end{proof}

\begin{remark}\label{ex:2x2matrixBrk}
In \cite{HKL13}, 
the secant
variety $X := \sigma_6(\bP^3\times\bP^3\times\bP^3)$
was considered.  The main theoretical result of \cite{HKL13} 
was constructing an exact polynomial vanishing on $X$
which was nonzero at $M_2$, the $2\times2$ matrix multiplication
tensor, thereby showing that the border rank of $M_2$ was at least~$7$.
However, before searching for such a polynomial, 
a version of the membership test described in 
Prop.~\ref{prop:BasicMembership} was used 
in \cite[\S~3.1]{HKL13} to show that $M_2\notin X$
by tracking \mbox{$\deg X = \hbox{15,456}$~paths.}
\end{remark}

We first illustrate Item~\ref{Item1} of 
Prop.~\ref{prop:BasicMembership} on two simple examples
and then, in Section~\ref{Ssec:BasicMembershipEx},
work through a more detailed example.

\begin{example}\label{Ex:CircleCusp}
To highlight the computation of the multiplicity, we consider
two illustrative examples in $\bC^3$:
$$\sJ_1 = \{(x,y,s)~|~x(1+s^2) = 2s, ~y(1+s^2) = 1-s^2\}
\hbox{~~and~~}
\sJ_2 = \{(x,y,s)~|~xs^2 = ys^3 = 1\}.$$
For $\pi(x,y,s) = (x,y)$, clearly
$J_1 := \overline{\pi(\sJ_1)} = \Var(x^2 + y^2 -1)$ 
and 
$J_2 := \overline{\pi(\sJ_2)} = \Var(x^3 - y^2)$.
Even though $P_1 := (0,-1)\notin\pi(\sJ_1)$ and $P_2 := (0,0)\notin\pi(\sJ_2)$, we can use Item~\ref{Item1} of Prop.~\ref{prop:BasicMembership} to show that $P_i\in J_i$
of multiplicity $i$ for $i = 1,2$.  

In particular, since $\deg J_1 = 2$, 
the membership test for $P_1$ with respect to $J_1$ 
requires tracking two paths on $\sJ_1$.  
The projection of the two paths limit to distinct points on $J_1$,
one of which is $P_1$.  Hence, the multiplicity of $P_1$ with
respect to $J_1$ is $1$, i.e., $P_1$ is a smooth point on $J_1$.  
Since $P_1\notin \pi(\sJ_1)$,
the path in $\sJ_1$ whose projection limits to $P_1$ does not 
have an endpoint in $\sJ_1$. 
This is illustrated in Figure \ref{figure:ExamplePaths}. 
\begin{figure}[!ht]
  \centering
  \includegraphics[width=0.5\textwidth]{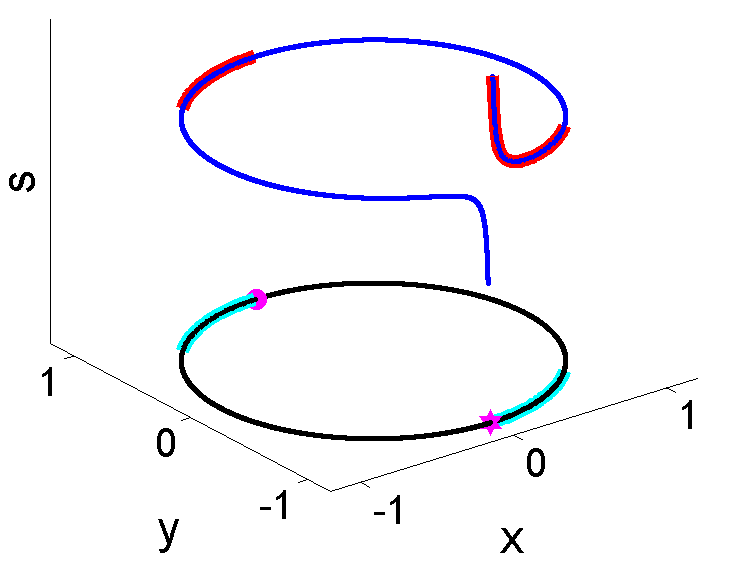}
  \caption{A plot of the two paths inside the curve $\sJ_1$,
  one of which diverges, and their projection into $J_1$ used to show that $P_1 = (0,-1)$, the six-pointed star, is a smooth point of $J_1$.
 }
\label{figure:ExamplePaths}
\end{figure}

Similarly, since $\deg J_2 = 3$, 
the membership test for $P_2$ with respect to $J_2$ 
requires tracking three paths on $\sJ_2$.  
The projection of the three paths limit to two distinct 
points on $J_1$ with the projection of 
two paths ending at $P_2$.  
Hence, the multiplicity of $P_2$ with
respect to $J_2$ is $2$.  Since $P_2\notin \pi(\sJ_2)$,
the two paths in $\sJ_2$ whose projection limits to $P_2$ do 
not have an endpoint in $\sJ_2$.

\end{example}

\subsection{Illustrative example using membership test}\label{Ssec:BasicMembershipEx}

To demonstrate various formulations that we can
utilize with Prop.~\ref{prop:BasicMembership}, 
we consider computing the border rank of the 
cubic polynomial $x^2y$ in $\Sym^3\bC^2$,
thereby verifying the results of
\cite[Table~1]{LT},
namely $\brk(x^2y) = 2$ (see also \cite{Syl,CS,bgi} for more general result).
The computation also yields that either $\rk(x^2y)>2$ 
or $\rk(x^2y) = 2$ with infinitely many decompositions.  
We will reconsider this example in Sections~\ref{Ssec:AdvancedMembership} and~\ref{Ssec:AllDecompositions}.
This subsection ends with a general discussion of Waring problems.

\subsubsection*{Border rank 1 test using affine cones}

We start our computation with the affine cone of the abstract join variety
built from a parameterization:
$$\sJ = \{(P,a)~|~P = v_3(a)\}\subset\bC^4\times\bC^2 \hbox{~~where~~}
v_3(a) = (a_1^3,3a_1^2a_2,3a_1a_2^2,a_2^3)$$
and $P=(P_{1},\ldots,P_{4})$.
With $\pi(P,a) = P$, the affine cone of elements of $\Sym^3\bC^2$ of border rank~$1$~is
$$J = \overline{\pi(\sJ)} = \overline{\{v_3(a)~|~a\in\bC^2\}}\subset\bC^4.$$
To compute $\dim J$, where $J$ is considered an affine variety in $\bC^4$, 
we use \eqref{eq:DimensionImage} with $\dim \sJ = 2$,
$B = [I_4~0]$, and $f(P,a) = P - v_3(a)$.  
It is easy to verify that $\dim \Null \left[\begin{array}{c} Jf(v_3(a^*),a^*) \\ B\end{array}\right] = 0$ for general $a^*\in\bC^2$ thereby showing $\dim J = \dim \sJ = 2$.

We construct a pseudowitness set for $J$, say $\{f,\pi,\sM,U\}$ where 
$\sM$ is defined by
$$\left[\begin{array}{c} P_1 - 3 P_3 - 5 P_4 - 2\\
P_2 + 2 P_3 + 4 P_4 - 3 \end{array}\right] = 0
\hbox{~~and~~}
U = \left\{
\begin{array}{l}
(-0.754,3.29,-4.78,2.32,-0.91\omega^k,1.32\omega^k), \\
(-6.01,7.99,-3.54,0.523,-1.82\omega^k,0.806\omega^k),\\
(3.39,2.06,0.416,0.028,1.50\omega^k,0.304\omega^k)
\end{array}\right\}
$$
for $k=0,1,2$ where $\omega$ is the third root of unity.
Hence, $\deg J = 3$ meaning that we can test membership by
tracking at most $3$ paths, say starting at $U'\subset U$ obtained with $k = 0$.

Since $x^2y$ corresponds to the point $P = (0,1,0,0)$, we consider
the linear space $\sL_P$ containing~$P$ defined by the linear equations
$$\left[\begin{array}{c}
P_1 + (2+\sqrt{-1}) P_3 - 3 P_4\\
(P_2-1) + 3 P_3 - (4+2\sqrt{-1}) P_4
\end{array}\right] = 0.$$
The projection under $\pi$ of the endpoints of the three paths 
derived from deforming $\sM$ to $\sL_p$ are
$$
\begin{array}{l}
(0.181+ 0.284\sqrt{-1}, 0.120+0.313\sqrt{-1}, 0.054 + 0.330\sqrt{-1}, -0.014+0.332\sqrt{-1}), \\
(-0.732 + 0.256\sqrt{-1}, 0.269+0.309\sqrt{-1}, 0.099 - 0.193\sqrt{-1}, -0.114 - 0.010\sqrt{-1}), \\
(-0.400 - 0.149\sqrt{-1}, 0.138 - 0.233\sqrt{-1}, 0.129+0.112\sqrt{-1}, -0.084+0.068\sqrt{-1}).
\end{array}$$
Since $P$ is not one of these three points, we know 
$1 < \brk(x^2y) \leq \rk(x^2y)$.

\subsubsection*{Border rank 2 test using affine cones}

Every polynomial in $\Sym^3\bC^2$ has border rank at most $2$.  
We can verify this simply using \eqref{eq:DimensionImage} with
\begin{equation}\label{Eq:BorderRank2Join}
\sJ = \{(P,a,b)~|~P = v_3(a)+v_3(b)\}\subset\bC^4\times\bC^2\times\bC^2.
\end{equation}
That is, for $\pi(P,a,b) = P$, we have
that $J = \overline{\pi(\sJ)} = \bC^4$.  
Since $\brk(x^2y) > 1$, this immediately shows that $\brk(x^2y) = 2$ which we can verify
by tracking one path defined by
\begin{equation}\label{eq:Membershipx2y}
P(t) = (1-t)(0,1,0,0) + t(-9\sqrt{-1},3-5\sqrt{-1},-1-11\sqrt{-1},-13) = v_3(a) + v_3(b).
\end{equation}
At $t = 1$, we can start with $a = (1+\sqrt{-1},1-2\sqrt{-1})$ and $b = (2-\sqrt{-1},1-\sqrt{-1})$.
One clearly has $\lim_{t\rightarrow0} P(t) = (0,1,0,0)$, 
but the corresponding $(a,b)$ diverge to infinity at $t\rightarrow0$.
Since starting from any of the possible decompositions of $P(1)$, 
namely 
$$P(1) = v_3(\omega^j a) + v_3(\omega^k b) = v_3(\omega^j b) + v_3(\omega^k a)$$
where $j,k = 0,1,2$ and $\omega$ is the third root of unity, yields a divergent path, 
Item~\ref{Item4} of Prop.~\ref{prop:BasicMembership}
states that either $\rk(x^2y) > 2$ or $\rk(x^2y) = 2$ with infinitely many decompositions.

We will focus on distinguishing between these two cases in Sections~\ref{Ssec:AdvancedMembership}
and~\ref{Ssec:AllDecompositions}.

\begin{remark}\label{RemSec} Once \eqref{eq:DimensionImage} has been
used to verify that a certain secant or join variety fill the
ambient space, this technique of tracking only one path can be used in general
as discussed in Section~\ref{Sec:LargeRank}.
Remarkably, it turns out that defective secant varieties 
are almost always those that one was expecting to fill the ambient space. 
\end{remark}

\subsubsection*{Border rank 2 test using affine coordinate patches}

We compare the behavior of using affine cones above with 
the use of an affine coordinate patch with scalars.
The advantage here is that, assuming sufficiently general coordinate
patches, all paths will converge with this setup.  
The paths
for which $\lim_{t\rightarrow0} Z_w(t)$ does 
not exist in $\sJ$ have the corresponding scalar, namely $\lambda_0$, equal to zero.  In particular, consider
$$Z = \left\{(P,\lambda,a,b)~\left|~\begin{array}{ll}P\lambda_0 = v_3(a)\lambda_1 + v_3(b)\lambda_2 \\
r_1(P) = r_2(\lambda) = r_3(a) = r_4(b) = 0 \end{array}\right\}\right.\subset\bC^4\times\bC^3\times\bC^2\times\bC^2$$
where each $r_i$ is a general affine linear polynomial.  
The irreducible component of interest inside of $Z$ 
is~$\overline{Z\setminus\Var(\lambda_0)}$.
Since this irreducible set plays a similar role
of the abstract join variety, we will call it $\sJ$.
With projection $\pi(P,\lambda,a,b) = P$, 
we have that $\overline{\pi(\sJ)} = \Var(r_1(P))$
which again verifies that every element 
in~$\Sym^3\bC^2$ has border rank at most $2$.

For concreteness and simplicity, we take 
$$r_1(P) = P_2-1,~~r_2(\lambda) = 2\lambda_0 - \lambda_1 + 3\lambda_2 - 1,~~r_3(a) = 3a_1 - 2a_2 - 1,~~\hbox{and}~~r_4(b) = 2b_1 + 3b_2 - 1.$$
Consider the path in $\overline{\pi(\sJ)} = \Var(r_1(P))$ 
defined by 
$P(t) = (1-t)(0,1,0,0) + t(1+\sqrt{-1},1,2-\sqrt{-1},1-2\sqrt{-1})$.
This path lifts to a path in $\sJ$, say, starting at $t=1$ with
$$\begin{array}{l}
\lambda = (0.00145+0.000914\sqrt{-1},0.0482+0.0524\sqrt{-1},0.348+0.0169\sqrt{-1}),\\ a = (0.21+0.0532\sqrt{-1},-0.184+0.0797\sqrt{-1}), \hbox{~~and~~} \\
b = (0.157+0.0666\sqrt{-1},0.229-0.0444\sqrt{-1}).\end{array}$$
As mentioned above, the advantage is that this path is convergent,
but the endpoint has $\lambda_0 = 0$ thereby showing that it would have diverged if we used an affine cone formulation. 

\subsubsection*{Waring problem}

This example of computing the rank and border rank 
of $x^2y$ in $\Sym^3\bC^2$
is an example of the so-called Waring problem, namely writing
a homogeneous polynomial as a sum of powers of linear forms.
We leave it as an open challenge problem to derive a 
general formula for the degrees
of the corresponding secant varieties for these problems
since this is equal to the maximum number of paths that need 
to be tracked in order to decide membership.
We highlight some of the known partial results.

The Veronese variety that parameterizes $d^{\rm th}$ 
powers of linear forms in $n+1$ variables is classically known 
to have degree $d^n$. 

In the binary case, the degree of $\sigma_2(X)$ where $X$ is the rational normal curve of degree~$d$ parametrizing forms of type $(ax+by)^d$ is ${d-1 \choose 2}$ \cite{AMS,ran}. More generally in \cite{cs} it is shown  that the variety parameterizing forms of type $(a_0x_0+ \cdots + a_nx_n)^d+(b_0x_0+\cdots +b_nx_n)^d$ has degree 
$$\frac{1}{2}\left( d^{2n}- \sum_{j=0}^n (-1)^{n-j}d^j{2n+1 \choose j}{2n-j \choose j} \right).$$
The same paper also computes the degree of $\sigma_2(X)$ where $X$ is any Segre-Veronese variety which parameterizes multihomogeneous polynomials of type $L_1^{d_1}\cdots L_k^{d_{k}}$ where $L_i$ is a linear form in the variables $x_{0,i}, \ldots , x_{n_i,i}$ for $i=1,\ldots,k$.  In this case, $\deg \sigma_2(X)$ is 
{\scriptsize
$$\frac{1}{2}\left( \left( (n_1, \ldots  , n_k)!d_1^{n_1}\cdots d_k^{n_k} \right)^2-\sum_{l=0}^n{2n+1 \choose l}(-1)^{n-l}\sum_{\sum j_i=n-l}(n_1-j_1, \ldots , n_k-j_k)! \prod_{i=1}^k{n_i+j_i\choose j_j}d_i^{n_i-j_i}\right).$$ }

The degree of some $k$-secants of ternary forms is known, e.g., \cite[Thm.~1.4]{ES} and \cite[Rem.~7.20]{IK}.

\subsection{Reduction to the curve case}\label{Ssec:AdvancedMembership}

Proposition~\ref{prop:BasicMembership} can determine membership
in join varieties as well as provide some insight regarding
membership in the constructible join.  In particular,  
Item~\ref{Item5} of Prop.~\ref{prop:BasicMembership}
shows that deciding membership in a join variety and the corresponding 
constructible join is equivalent when the join variety is a curve.
The following describes one approach for deciding membership
in the constructible join by reducing down to the curve case.
Section~\ref{Ssec:AllDecompositions} considers
computing all decompositions of the form \eqref{eq:JoinDefEq}
and hence can also be used to decide membership in the constructible
join by simply deciding if such a decomposition exists.

Suppose that $\sJ$ is the abstract join with corresponding join
variety $J = \overline{\pi(\sJ)}$.  Since we want to reduce down to
the curve case, we will assume that $d := \dim J > 1$. 
Let $C$ be a general curve section of $J$, that is, $C = J\cap \sL$
where $\sL$ is a general linear space with $\codim \sL = d - 1$.  
Since~$J$ is irreducible and $\sL$ is general, 
the curve $C$ is also irreducible.
Hence, $\sJ_C = \pi^{-1}(C)\cap\sJ$ is irreducible with $C = \overline{\pi(\sJ_C)}$.
Therefore, one can use Item~\ref{Item5} of Prop.~\ref{prop:BasicMembership} to test
membership in $C$ and $C^0 = \pi(\sJ_C)$.
However, since $\sL$ is general, 
testing membership in $C^0$ and $C$ is typically 
not the problem of interest.

Given $[P]$, we want to decide if $[P]$ is a member of $J^0$ or $J$.  Thus,
we could modify the description above to replace $\sL$ with $\sL_P$, a general linear
slice of codimension $d-1$ passing through $[P]$.  If $C_P = J\cap\sL_P$, 
then $\sJ_{C_P} = \pi^{-1}(C_P)\cap\sJ$ need not be irreducible. 
However, since~$\sL_P$ is general through $[P]$, 
the closure of the images under $\pi$ of each irreducible 
component of~$\sJ_{C_P}$ must either be the singleton $\{[P]\}$
or the curve $C_P$.  
Thus, one can apply Item~\ref{Item5} of Prop.~\ref{prop:BasicMembership}
to each irreducible component of $\sJ_{C_P}$ whose image
under~$\pi$~is~$C_P$.

The following illustrates this reduction to the curve case.

\begin{example}\label{Ex:Rank3}
Consider $\sJ$ as in \eqref{Eq:BorderRank2Join} with $J = \overline{\pi(\sJ)} = \bC^4$
and $d = 4$.  Since a general curve section of $J$ is simply a general
line in $\bC^4$, we have $C = \sL$ where $\sL\subset\bC^4$ is a general line.  
It is easy to verify that $\sJ_C = \pi^{-1}(C)\cap\sJ$ is an irreducible curve of degree 30.  

We now consider the point $P$ corresponding to~$x^2y$, namely $(0,1,0,0)$.  
Hence, $C_P = \sL_P$ where $\sL_P$ is a general line through this point.
In this case, $\sJ_{C_P} = \pi^{-1}(C_P)\cap\sJ$ is also an irreducible curve of
degree~30.  Hence, we can apply Item~\ref{Item5} of Prop.~\ref{prop:BasicMembership}
to decide membership of $x^2y$ in $J^0 = \pi(\sJ)$ by deciding membership in 
$C_P^0 = \pi(\sJ_{C_P})$.  Since $C_P^0$ is a line, this is equivalent
to tracking paths defined between a general point and~$P$, as
in \eqref{eq:Membershipx2y}.  Since all paths diverge, 
Item~\ref{Item5} of Prop.~\ref{prop:BasicMembership} yields $\rk(x^2y) > 2$.  

For comparison, suppose that we want to consider $C_{Q} = \sL_Q$,
which is a general line through $Q = (1,3,3,1)$ corresponding to
$x^3 +3\, x^2y + 3\, xy^2 + y^3$ which has rank one.
The curve $\sJ_{C_Q} = \pi^{-1}(C_Q)\cap\sJ$ is the union of two irreducible curves,
say $V_1$ and $V_2$ with $\pi(V_1) = \{Q\}$ and $\overline{\pi(V_2)} = C_Q$,
both of which yield decompositions.
\end{example}

\subsection{Computing all decompositions}\label{Ssec:AllDecompositions}

A fundamental question related to rank is to describe the set of all of 
decompositions of a point~$[P]$.  In numerical algebraic geometry,
this means computing a {\em numerical irreducible decomposition}, i.e.,
a witness set for each irreducible component, 
of the fiber over $[P]$, namely
$$\sF_P := \pi^{-1}([P])\cap \sJ(X_1,\dots,X_k).$$
Computing $\sF_P$ yields a membership test for $J^0(X_1,\dots,X_k)$ since 
$[P]\in J^0(X_1,\dots,X_k)$ if and only if~$\sF_P\neq\emptyset$.
One approach is to directly compute a numerical irreducible decomposition
using~\eqref{eq:JoinDefEq}. 
Another approach is to perform a cascade \cite{RegenCascade,Cascade}
starting with a witness set for~$\sJ$.  
Since $\pi^{-1}([P])$ is defined by linear equations, computing $\sF_P$
can be simply obtained by degenerating each general slicing hyperplane
to a general hyperplane containing $\pi^{-1}([P])$.  After each
degeneration, the resulting points are either contained in $\pi^{-1}([P])$,
forming {\em witness point supersets}, or not.  The ones not contained
in $\pi^{-1}([P])$ are used as the start points for the next degeneration.  
This process is described in detail in \cite[\S~2.2]{HW13}.
From the witness point supersets, standard methods in numerical algebraic
geometry (e.g., see \cite[Chap.~10]{BertiniBook}) are used to produce 
the numerical irreducible decomposition of $\sF_P$.  

After determining that $[P]\in J^0(X_1,\dots,X_k)$ by showing $\sF_P\neq\emptyset$,
a numerical irreducible decomposition of $\sF_P$
can then be used to perform additional computations.  
One such computation is deciding if $\sF_P$ contains
real points, i.e., determining if there is a real decomposition,
which is described in Section~\ref{Sec:RealRank}.  
Another application is to determine
if there exist ``simpler'' decompositions of $[P]$, e.g., 
deciding if \mbox{$[P]\in J^0(X_1,\dots,X_{k-1})$}.  
In the secant variety case, this is equivalent to deciding if 
the rank of $[P]$ is strictly less than $k$.
The following illustrates this idea continuing with $x^2y$
considered in Section~\ref{Ssec:BasicMembershipEx}.

\begin{example}\label{Ex:rank3}
With the setup from Section~\ref{Ssec:BasicMembershipEx},
consider computing all of the rank $3$ decompositions of $x^2y$
using affine cones.  That is, we consider
\begin{equation}\label{Eq:BorderRank3Join}
\sJ = \{(P,a,b,c)~|~P = v_3(a)+v_3(b)+v_3(c)\}\subset\bC^4\times\bC^2\times\bC^2\times\bC^2 = \bC^{10}
\end{equation}
which is irreducible of dimension 6 and degree 57.  Thus, in a witness set for $\sJ$,
we have a general linear space $\sL$ of codimension 6 defined by 
linear polynomials $\ell_i(P,a,b,c) = 0$, $i = 1,\dots,6$
and a witness point set $W = \sJ\cap\sL$ consisting of 57 points.

For $i = 1,\dots,4$, let $m_i(P)$ be a general linear polynomial which vanishes at $(0,1,0,0)$.
The cascade simply replaces the condition $\ell_i = 0$
with $m_i = 0$ sequentially for $i = 1,\dots,4$.  
For $i = 1,2,3$, we have that
$\sJ\cap\Var(m_1,\dots,m_i,\ell_{i+1},\dots,\ell_6)$
consists of 57 points, none of which are not contained in $\sF_P$.
However, $\sJ\cap\Var(m_1,\dots,m_4,\ell_5,\ell_6)$
consists of 45 points, all of which are contained in $\sF_P$.
Hence, $\sF_P\neq\emptyset$ thereby showing that $\rk(x^2y) \leq 3$.
In fact, these 45 points form a witness point set for $\sF_P$,
which is an irreducible surface of degree $45$.  

Even though Ex.~\ref{Ex:Rank3} shows that $\rk(x^2y) > 2$, we 
can verify this by showing that $\sF_P\cap\Var(c)=\emptyset$ 
using the witness point set for $\sF_P$ computed above.
To that end, let $r_i(c)$ for $i = 1,2$ 
be a general linear polynomial vanishing at $c = 0$.
By deforming from the general linear $\ell_5(P,a,b,c)$ to $r_1(c)$,
we obtain that $\sJ\cap\Var(m_1,\dots,m_4,r_1,\ell_6)$ consists of 
36 points, none of which satisfy $c = 0$.  We then deform $\ell_6$ to $r_2$ thereby computing 
$\sJ\cap\Var(m_1,\dots,m_4,r_1,r_2)$, which is empty.
Hence, $\sF_P\cap\Var(c) = \emptyset$ so that 
$\rk(x^2y) > 2$. 
\end{example}

\section{Boundary}\label{Sec:Boundary}

By using the various approaches described in Section~\ref{Sec:Membership},
one is able to use numerical algebraic geometry to determine membership
in both the constructible join $J^0(X_1,\dots,X_k)$ and the join variety $J(X_1,\dots,X_k)$.  An interesting
object is the {\em boundary} 
$\partial := \overline{J(X_1,\dots,X_k)\setminus J^0(X_1,\dots,X_k)}$
which is the closure of the elements which arose by closing
the constructible set $J^0(X_1,\dots,X_k)$.
As a subset of $J(X_1,\dots,X_k)$, the boundary~$\partial$ 
may consist of irreducible components of various codimension.  In the following, 
we describe an approach for computing the irreducible components of $\partial$
which have codimension one with respect to $J(X_1,\dots,X_k)$, denoted $\partial_1$,
derived from~\mbox{\cite[\S~3]{HS13}}.

Following with the numerical algebraic geometric framework,
we aim to compute a pseudowitness set for $\partial_1$.  
To do this, we first consider the case 
where $\sC\subset\bP^N\times\bC^M$
is an irreducible curve and the projection $\pi([P],X) = [P]$ is generically finite-to-one on $\sC$,
i.e.,~$\dim \sC = 1$ and $\dim_{gf}(\sC,\pi) = 0$.  
The boundary of $C = \pi(\sC)$ consists of at most finitely many points
$\partial_C = \overline{C}\setminus C$.

\begin{example}\label{ex:EasyBdry}
Consider the irreducible curve $\sC = \{([a,b],x)~|~a\cdot x = b\}\subset\bP^1\times\bC$
and $\pi([a,b],x) = [a,b]$.  Generically, $\pi$ is a one-to-one map
from $\sC$ to $\bP^1$.  Since we have $x = b/a\in\bC$ when $a\neq 0$, one can 
easily verify that the boundary of $C = \pi(\sC)$ is $\partial_C = \{[0,1]\}\subset\bP^1$.
\end{example}

The first task in computing $\partial_C$ is to compute a superset
of $\partial_C$ consisting of finitely many points.
To that end, consider the closure of 
$\sC\subset\bP^N\times\bC^M$ 
in $\bP^N\times\bP^M$, say $\overline\sC$.
Then, a finite superset of $\partial_C$ is the set of points in $\overline C$
whose fiber intersects ``infinity.''  That is, if we have coordinates $x\in\bC^M$
and $[y]\in\bP^M$ with the embedding given by
$$(x_1,\dots,x_M)\in\bC^M\mapsto[1,x_1,\dots,x_M]\in\bP^M,$$
then a finite superset of $\partial_C$ is $\pi(\overline\sC\cap\Var(y_0))$.

\begin{example}\label{ex:EasyBdry2}
Continuing with the setup from Ex.~\ref{ex:EasyBdry}, one can verify that 
$$\overline\sC = \{([a,b],[y_0,y_1])~|~a\cdot y_1 = b \cdot y_0\}\subset\bP^1\times\bP^1
\hbox{~~and~~}\pi(\overline\sC\cap\Var(y_0)) = \{[0,1]\}.$$
\end{example}

In Ex.~\ref{ex:EasyBdry2}, it was the case that 
$\partial_C = \pi(\overline\sC\cap\Var(y_0))$.  However,
since $\partial_C \subset\pi(\overline\sC\cap\Var(y_0))$ in general,
we must investigate each point in $\pi(\overline\sC\cap\Var(y_0))$
via Sections~\ref{Ssec:AdvancedMembership} and~\ref{Ssec:AllDecompositions} 
to determine if it is contained in $\partial_C$.
\medskip

With the special case in hand, we now turn to the 
general case of $\sJ$ as in \eqref{eq:TotalSpace}.
Let $J$ and $J^0$ be the corresponding join variety and constructible join, 
i.e., $J = \overline{\pi(\sJ)}$ and $J^0 = \pi(\sJ)$,
and $d = \dim J$.  Since the case $d = 0$ is trivial, we will assume $d \geq 1$.
Since we aim to compute witness points sets for the 
codimension~1 components, $\partial_1$, of the boundary 
$\partial = \overline{J\setminus J^0}$, i.e., $\partial_1$
has pure-dimension $d-1$, we can restrict our attention to a general curve
section of $J$, say $C$.  This restriction cuts $\partial_1$
down to finitely many points, i.e., a witness point set for $\partial_1$, which we aim to compute.

Since $C$ is a general curve section,
$\sM = \pi^{-1}(C)\cap\sJ$ is irreducible.  Finally, we take a general curve
section of $\sM$, say $\sC$.  Thus, $\sC$ is an irreducible 
curve with $\overline{\pi(\sC)} = C$ and $\dim_{gf}(\sC,\pi) = 0$.  
Applying the procedure described above yields 
a finite set of points containing $\partial_C$.  
Since the restriction from $\sM$ to a general curve section $\sC$ may have
introduced new points in the boundary, we simply need to investigate
each point with respect to $\sM$ rather than $\sC$
via Sections~\ref{Ssec:AdvancedMembership} and~\ref{Ssec:AllDecompositions}.
In the end, we obtain the finitely many points forming a witness point
set for $\partial_1$. 

\begin{example}\label{Ex:AdvancedMembershipEx}
As with Section~\ref{Ssec:BasicMembershipEx}, we use
a parameterization to compute the codimension one component
of the boundary, $\partial$, in $\Sym^3\bC^2$ of border rank $2$.
Since every polynomial in $\Sym^3\bC^2$ has border rank $2$,
the codimension one component of $\partial$ is a hypersurface: the tangential variety of the rational normal cubic curve.

Since $J = \Sym^3\bC^2$, a general curve section $C$ of $J$ is 
simply a general line.  Following an affine cone formulation as in \eqref{Eq:BorderRank2Join}, 
we take, for exposition, $C$ defined by the equations
$$P_1 + 3P_4 - 2 = P_2 - 4P_4 + 3 = P_3 - 2P_4 - 4 = 0.$$
Since $\dim_{gf}(\sJ,\pi) = 0$, 
we have $\sC = \sM = \pi^{-1}(C)\cap\sJ$ is the curve
$$\sC = \{(P,a,b)~|~P = v_3(a) + v_3(b) \in C\}\subset\bC^4\times\bC^2\times\bC^2.$$
Next, we compute the closure, $\overline\sC$, of $\sC$ in $\bC^4\times\bP^4$
where $\bC^2\times\bC^2\hookrightarrow\bP^4$ given by
$$(a,b)\in\bC^2\times\bC^2 \mapsto [1,a_1,a_2,b_1,b_2]\in\bP^4.$$
With coordinates $[y_0,\dots,y_4]\in\bP^4$, we find that 
$\pi(\overline\sC\cap\Var(y_0))$ consists of the following four points:
$$\begin{array}{ll}
(2.308,-3.410,3.794,-0.103), & (-35.743,47.325,29.163,12.581), \\
(4.328,-6.103,2.448,-0.776), & (0.018,-0.357,5.321,0.661). \end{array}$$
Finally, we verified that each of these points 
corresponds to an element that has rank larger than $2$
via Sections~\ref{Ssec:AdvancedMembership} and~\ref{Ssec:AllDecompositions}.  Hence,
the codimension one component of $\partial$, namely $\partial_1$, 
is a degree $4$ hypersurface.  

Although we can use numerical algebraic geometry to test membership in 
this hypersurface via Section~\ref{Sec:Membership}, 
we can also easily recover the defining equation
exactly in this case using \cite{Exactness}:
$$P_1^2 P_4^2 - 6 P_1 P_2 P_3 P_4 + 4 P_1 P_3^3 + 4 P_2^3 P_4 - 3 P_2^2 P_3^2 = 0.$$
Clearly, $(0,1,0,0)$, corresponding to $x^2y$, lies on this hypersurface.
\end{example}

\section{Real decompositions}\label{Sec:RealRank}

For a real $[P]$, i.e., the line $[P]$ is defined by linear polynomials with real coefficients, 
a natural question is to determine if {\em real} decompositions exist
after computing the fiber $\sF_P$ as in Section~\ref{Ssec:AllDecompositions} showing that decompositions over the
complex numbers exist.
With a witness set for each irreducible component of $\sF_P$,
where all general choices involve selecting real numbers,
the homotopy-based approach of \cite{H13}
(see also~\cite{WuReid}) can be used to determine if 
the irreducible component contains real points.
The techniques described in~\cite{H13,WuReid} rely
upon computing critical points of the distance function
as proposed by Seidenberg~\cite{Seidenberg} (see~also~\cite{RealSolving,EuclideanDistance,Real}). 
For secant varieties, this yields a method to determine the real rank
of a real element.

Let $F$ be a system of $n$ polynomials in $N$ variables with real coefficients and
\mbox{$V\subset\Var(F)\subset\bC^N$} be an irreducible component.  Fix
a real point $x^*\in\bR^N$ such that $x^*\notin\Var(F)$ sufficiently general.  Following Seidenberg~\cite{Seidenberg},
we consider the optimization problem
\begin{equation}\label{eq:Min}
\min\{\|x-x^*\|_2^2~|~x\in V\cap\bR^N\}.
\end{equation}
Every connected component $C$ of $V\cap\bR^N$ has a global minimizer of
the distance functions from $x^*$ to $C$, i.e., there exists $x\in C$ such that
$\|x-x^*\|_2^2 \leq \|z-x^*\|_2^2$ for every $z\in C$.  
Thus, there exists $\lambda\in\bP^n$ with
\begin{equation}\label{eq:CritPts}
G(x,\lambda) = \left[\begin{array}{c} F(x) \\ 
\lambda_0 (x-x^*) + \sum_{i=1}^n \lambda_i \grad F_i(x) \end{array}\right] = 0
\end{equation}
where $\grad F_i$ is the gradient of $F_i$.  
For the projection map $\pi(x,\lambda) = x$, the set $\pi(\Var(G))\subset\bC^N$
is called the set of {\em critical points} of~\eqref{eq:Min} and it intersects
every connected component of $V\cap\bR^N$.  Hence, $V\cap\bR^N = \emptyset$
if and only if there are no real critical points for~\eqref{eq:Min}.

This method allows one to decide if a real decomposition exists
by deciding if there exists a real critical point
of the distance function.  As a by-product, one
obtains the closest decomposition to the given point.
When the set of critical points may be positive-dimensional, 
the approach presented in \cite{H13} uses a homotopy-based approach
to reduce down to testing the reality of finitely many critical points.
Therefore, the problem of deciding if a real decomposition exists can 
be answered by deciding the reality of finitely many points. 

\begin{example}\label{Ex:rank3real}
Consider deciding if the real rank of $x^2y$ in $\Sym^3\bC^2$ 
is the same as the complex rank, namely~$3$.
In Ex.~\ref{Ex:rank3}, we computed $\sF_P$, which is irreducible
of dimension $2$ and degree $45$.  In particular, we can take 
$$F(a,b,c) = \left[\begin{array}{c} a_1^3 + b_1^3 + c_1^3 \\ 3(a_1^2a_2 + b_1^2b_2 + c_1^2c_2) - 1 \\
3(a_1a_2^2 + b_1b_2^2 + c_1c_2^2) \\ a_2^3 + b_2^3 + c_2^3 \end{array}\right].$$
We aim to compute the critical points of the distance from 
\begin{equation}\label{eq:ABCdistance}
\alpha = (1,2), ~~ \beta = (-2,1), ~~\gamma = (1,-1)
\end{equation}
which arise from the solutions $(a,b,c)\in\bC^6$ and $\delta\in\bP^4$ of
$$G(a,b,c,\delta) = \left[\begin{array}{c}
F(a,b,c) \\
\delta_0(a-\alpha,b-\beta,c-\gamma) + \sum_{i=1}^4 \delta_i\grad F_i(a,b,c)\end{array}\right].$$
Solving $G = 0$ yields 234 critical points, of which 8 are real.  
Hence, the real rank of $x^2y$ is indeed $3$ (cfr.~\cite{bcg}) where the one of minimal distance
from $(\alpha,\beta,\gamma)$ is the decomposition (to three digits)
\begin{equation}\label{Eq:MinDist}
x^2y = (0.721 x + 0.2849 y)^3 + (-1.429 x  + 1.101 y)^3 + (1.365 x - 1.107 y)^3.
\end{equation}
\end{example}

Since minimizing the distance to a nongeneric point
can yield potential issues, one should treat the center 
point $x^*$ as a parameter and utilize a parameter homotopy (e.g., see \cite[Chap.~6]{BertiniBook}).  The number of such paths
to track with this setup is called the 
Euclidean distance degree \cite{EuclideanDistance}.

\begin{example}\label{ex:CircleImage}
Consider solving the following optimizaton problem
$$\min\{x^2+(y+2)^2~|~(x,y)\in\overline{\pi(\sJ_1)}\cap\bR^2\}$$
where $\sJ_1$ and $\pi$ as in Ex.~\ref{Ex:CircleCusp}.
Since $\overline{\pi(\sJ_1)} = \Var(x^2 + y^2 - 1)$,
it is clear that the two critical points are $(0,\pm 1)$
with $(0,-1)\notin\pi(\sJ_1)$.  Thus, we consider this
as a member of the family of optimization problems
$$\min\{(x-q_1)^2+(y-q_2)^2~|~(x,y)\in\overline{\pi(\sJ_1)}\cap\bR^2\}$$
parameterized by $q$.  The critical points are obtained by solving
$$\left[\begin{array}{c} x(1+s^2) - 2s \\ y(1+s^2)-(1-s^2) \\
\lambda_0\left[\begin{array}{c} x - q_1 \\ y - q_2 \\ 0 \end{array}\right] + \lambda_1\left[\begin{array}{c} 1+s^2 \\ 0 \\ 2sx - 2 \end{array}\right] + \lambda_2\left[\begin{array}{c} 0 \\ 1+s^2 \\ 2sy + 2s \end{array}\right]\end{array}\right] = 0$$
which, for a general $q\in\bC^2$, has two solutions.
A parameter homotopy that deforms the parameters
from the selected value of $q$ to $(0,-2)$ yields solution two paths.
As paths in $\bC^3\times\bP^2$, only one of these two paths has a limit point since $(0,-1)\notin\pi(\sJ_1)$.
However, since we are interested in $\overline{\pi(\sJ_1)}$,
we only need to observe the limit
of the projection of these two paths in $(x,y)\in\bC^2$,
which yields the two critical points $(0,\pm 1)$.
\end{example}

The computation of {\em all} critical points provides a global approach
for deciding if a real decomposition exists.  
Such a global approach may be computationally expensive when the
number of critical points is large.  Thus, we also describe a local
approach based on {\em gradient descent homotopies} \cite{GH13}
with the aim of computing a real critical point.  
Although there is no guarantee, such a local approach can 
provide a quick affirmation that a real decomposition exists.

With the setup as above, we consider the gradient descent homotopy
$$H(z,\delta,t) = \left[\begin{array}{c} F(z) - t\cdot F(x^*)\\ 
\delta_0 (z-x^*) + \sum_{i=1}^n \delta_i \grad F_i(z) \end{array}\right] = 0.$$
Clearly, $H(x,\lambda,0) = G(x,\lambda) = 0$ for $G$ as in \eqref{eq:CritPts}.
We consider the homotopy path $(z(t),\delta(t))$ 
where $z(1) = x^* \in \bR^N$ and $\delta(1) = [1,0,\dots,0]\in\bP^n$.
If this homotopy path is smooth for $0 < t \leq 1$
and converges as $t\rightarrow0$, then $z(0)$ is a real critical point
with respect to $F$.  We note that $H$ is a 
so-called {\em Newton homotopy} 
since $\partial H/\partial t$ is independent of $z$ and $\delta$.
Newton homotopies will also be used below in Section~\ref{Sec:LargeRank}.

One can quickly try gradient descent homotopies for various $x^*\in\bR^N$ 
with the goal of computing a real critical point, provided one exists.

\begin{example}\label{Ex:rank3real2}
With the setup from Ex.~\ref{Ex:rank3real} and \eqref{eq:ABCdistance}, we consider the
gradient descent homotopy
$$H(a,b,c,\delta,t) = \left[\begin{array}{c}
F(a,b,c) - t\cdot F(\alpha,\beta,\gamma) \\
\delta_0(a-\alpha,b-\beta,c-\gamma) + 
\sum_{i=1}^4 \delta_i\cdot\nabla F_i(a,b,c)
\end{array}\right].$$
The path, which starts at $t = 1$ with $(\alpha,\beta,\gamma,[1,0,\dots,0])$
yields a smooth and convergent path with the endpoint corresponding
the decomposition of minimal distance from $(\alpha,\beta,\gamma)$ in \eqref{Eq:MinDist}.
\end{example}

\section{Generic cases}\label{Sec:LargeRank}

When the join variety $J(X_1,\dots,X_k)$ fills the ambient
projective space, the degree of the join variety is~$1$.  
In this section, we modify the approach
presented in Prop.~\ref{prop:BasicMembership}
to use a Newton homotopy which can compute a decomposition of a 
generic tensor by tracking one path.
Such paths can even be tracked certifiably \cite{HHL14,HL14}.

Let $[P]\in J(X_1,\dots,X_k)$ be generic.
Thus, the dimension and degree of the fiber over $[P]$
is the same over a nonempty Zariski open subset of $J(X_1,\dots,X_k)$,
i.e., 
$$d := \dim_{gf}(\sJ,\pi) = \dim \pi^{-1}([P])\cap\sJ
\hbox{~~and~~} \deg_{gf}(J,\pi) = \deg \pi^{-1}([P])\cap\sJ.$$
The first step for computing a decomposition of $[P]$
is to produce a starting point.  This is performed by 
selecting generic $x_i^*\in \Cone{X_i}$ and computing $P^* = \sum_{i=1}^k x_i^*$.  
That is, $([P^*],x_1^*,\dots,x_k^*)\in\sJ$ is sufficiently generic
with respect to $\sJ$ and $[P]$.  

With this setup, consider the homotopy that deforms
the fiber as we move along the straight line from $[P^*]$ to $[P]$,
namely $\pi^{-1}(t[P^*] + (1-t)[P])\cap\sJ$.  
If $d > 0$, i.e., the fiber is positive-dimensional,
we can reduce down to tracking along a path 
by simply intersecting with a generic linear space $\sL$ 
of codimension $d$ passing through the point $([P^*],x_1^*,\dots,x_k^*)$.
This results in the Newton homotopy
$$\pi^{-1}(t[P^*] + (1-t)[P])\cap\sJ\cap\sL$$
where, at $t = 1$, we have start point $(x_1^*,\dots,x_k^*)$.  
The endpoint of this path yields a decomposition of $[P]$ in 
the form \eqref{eq:JoinDefEq}.

\subsection{Illustrative example}\label{Ssec:GenericDecomp}

We demonstrate decomposing a general element via tracking
one path on cubic forms in $3$ variables.
For a cubic form $C(x)$, we aim to write it as 
\begin{equation}\label{Eq:DecompCubic}
C(x) = Q(x)\cdot L_1(x) + L_2(x)^3
\end{equation}
where $Q(x) = q_0 x_0^2 + q_1 x_0 x_1 + \cdots + q_5 x_2^2$ is a quadratic form and $L_i(x) = x_0 + a_{i1}x_1 + a_{i2}x_2$ is a linear form with $q_k,a_{ij}\in\bC$. 
Geometrically, this means $C(x)\in J(\mathcal{O}^2(\nu_3(\mathbb{P}^2)), \nu_3(\mathbb{P}^2))$ 
where $\mathcal{O}^2(\nu_3(\mathbb{P}^2))$ is the second osculating variety to the Veronese 
surface $\nu_3(\mathbb{P}^2)$.
By \eqref{eq:DimensionImage}, it is easy to verify that
a general cubic $C(x)$ must have finitely many decompositions of the form~\eqref{Eq:DecompCubic}.  
We will compute decompositions~of
$$C_1(x) = x_0^3 + x_1^3 + x_2^3 \hbox{~~~~and~~~~} C_2(x) = 4x_0^2x_2+x_1^2x_2-8x_0^3$$
where the cubic $C_2$ defines a curve called the ``witch of Agnesi.''
Starting with
$$\begin{array}{rcl}
C^*(x) &=& 
\left(x_0^2 + (1+\sqrt{-1})x_0x_1 + 3x_0x_2 - 2 x_1^2 + (3-\sqrt{-1})x_1x_2 + 2x_2^2\right)(x_0 + 2x_1 + 3x_2) \\
& & ~~~~+~\left(x_0-(3+\sqrt{-1})x_1+5x_2\right)^3,
\end{array}$$
the Newton homotopy deforming from $C^*$ to $C_i$ which is obtained by taking coefficients of
$$tC^*(x) + (1-t)C_i(x) = (q_0x_0^2 + q_1 x_0 x_1 + \cdots + q_5x_2^2)(x_0 + a_{11}x_1 + a_{12}x_2) + (x_0 + a_{21}x_1 + a_{22}x_2)^3$$
yields the following decompositions, which
we have converted to exact representation using \cite{Exactness}:
$$
\begin{array}{rcl}
C_1(x) &=& (-3x_0x_1 + 3(1-\sqrt{-3})x_0x_2/2 - 3x_1^2 + 3(1-\sqrt{-3})x_1x_2/2)(x_0-(1-\sqrt{-3})x_2/2) \\
& & ~~~~+~(x_0 + x_1 - (1-\sqrt{-3})x_2/2)^3,\medskip\\
C_2(x) &=& (-9x_0^2 - 9x_1^2/4)(x_0 - \sqrt{3} x_1/6 - 4 x_2/9) + (x_0 - \sqrt{3} x_1/2)^3. \\
\end{array}
$$

\subsection{Projections and Newton homotopies}\label{Ssec:LocalNewton}

When the join variety does not fill the ambient space,
we will modify our approach by combining Newton homotopies,
projections, and local numerical solving techniques, e.g., 
\cite{KBSIAMReview,Smirnov}.  This yields a local method
which can be used to show upper bounds on rank and border
rank over $\bC$ and $\bR$.  

If $c := \codim J(X_1,\dots,X_k) > 0$, let $\psi$ 
be a linear map such that $\overline{\psi(J(X_1,\dots,X_k))}$ 
fills the ambient space and $\dim_{gf}(J(X_1,\dots,X_k),\psi) = 0$.
We could apply the method above to attempt
to decompose $\psi([P])$ using a Newton
homotopy in $\overline{\psi(J(X_1,\dots,X_k))}$
starting from a randomly selected point.  
Since the fiber $\psi^{-1}(\psi([P]))\cap J(X_1,\dots,X_k)$ 
may contain many other elements in addition to $[P]$,
we may need to run the Newton homotopy approach
with many starting points to increase our
chances of ending at a decomposition of $[P]$.

Another approach is to use \cite{HOOS}
to compute all of the elements in a fiber
over a general element $\psi([P^*])$ and then use
the Newton homotopy to track all of the corresponding paths 
as $\psi([P^*])$ is deformed to $\psi([P])$.  
If $\psi$ was generic with respect to $[P]$ so that the fiber 
$\psi^{-1}(\psi([P]))\cap J(X_1,\dots,X_k)$ 
is zero dimensional, then tracking all of these paths
and observing in any end at $[P]$ is equivalent to
Item~\ref{Item1} of Prop.~\ref{prop:BasicMembership}.

Rather than utilizing a global method, 
we will use local numerical decomposition methods, e.g., 
see~\cite{KBSIAMReview,Smirnov}, to ``seed'' our Newton homotopy.
Such local methods generally use optimization techniques 
to numerically approximate a decomposition, 
and these approximations typically provide excellent
starting points.  Moreover, if the homotopy is defined 
by polynomials with real coefficients
and the start point is real, then the endpoint is also real
provided that the path is smooth on $(0,1]$.  This observation
allows one to yield upper bounds on the {\em real} rank and border rank by demonstrating the existence of {\em real} paths.

\begin{example}\label{ex:LocalNewton}
Reconsider $\sJ_1$ from Ex.~\ref{Ex:CircleCusp} where
$J_1 := \overline{\pi(\sJ_1)} = \Var(x^2+y^2-1)\subset\bC^2$.
For the projection map $\psi(x,y) = 2x + 3y$, it
is easy to verify that $\overline{\psi(J_1)} = \bC$
with $\dim_{gf}(J_1,\psi) = 0$ and $\deg_{gf}(J_1,\psi) = 2$.
We aim to show that there is a real path $(x(t),y(t),s(t))\in \sJ_1$
such that $\lim_{t\rightarrow0} \pi(x(t),y(t)) = (0,-1)$.  

For this simple example, we can easily construct a point nearby $(0,-1)$ which has a real point in its fiber, say $(x^*,y^*,s^*)=(20/101,-99/101,10)\in\sJ_1$, which we take as the start point for the Newton homotopy
$$H(x,y,s,t) = \left[\begin{array}{c}
x(1+s^2) - 2s \\ 
y(1+s^2) - (1-s^2) \\
\psi(x,y) - (t\cdot\psi(x^*,y^*) + (1-t)\cdot\psi(0,-1))
\end{array}\right].$$
One can easily observe that the path $(x(t),y(t),s(t))$
with $(x(1),y(1),s(1)) = (x^*,y^*,s^*)$ is smooth
on $(0,1]$ with $(x(t),y(t))\rightarrow(0,-1)$ as $t\rightarrow0$.
Since $(0,-1)\notin\pi(\sJ_1)$, this path must diverge in $\sJ_1$.

For numerical computations, it can be easier to track
convergent paths.  In this case, one can compactify
the fiber with respect to $\pi$ to yield
$$H(x,y,s,t) = \left[\begin{array}{c}
x(s_0^2+s_1^2) - 2s_0s_1 \\ 
y(s_0^2+s_1^2) - (s_0^2-s_1^2) \\
\psi(x,y) - (t\cdot\psi(x^*,y^*) + (1-t)\cdot\psi(0,-1))
\end{array}\right]$$
with start point $x = x^*$, $y = y^*$, and $s = [1,s^*]$
at $t = 1$.  Thus, one is actually tracking the path
on the closure of $\sJ_1$ in $\bC^2\times\bP^1$.
The endpoint of this smooth path on $(0,1]$
is $(x,y) = (0,-1)$ with $s = [0,1]$.
\end{example}

\section{Examples}\label{Sec:Examples}

The previous sections have described various approaches for
computing information about join and secant varieties along with several
illustrative examples.  In this section, we present 
several larger examples which were computed using the 
methods described above with computations facilitated by {\tt Bertini} \cite{Bertini,BertiniBook}.

\subsection{Complex multiplication tensor}\label{Sec:ComplexMult}

Complex multiplication can be treated as a bilinear map from $\bR^2\times\bR^2\rightarrow\bR^2$, namely
$$(a,b)\cdot(c,d)\mapsto(ac-bd,ad+bc),$$
which, using the definition, involves $4$ multiplications. 
Treating this as a bilinear map from $\bC^2\times\bC^2\rightarrow\bC^2$, we will use the above approaches to compute the rank and border rank (over $\bC$) of this bilinear map, both of which are $2$. 
We will then demonstrate how our method shows that the real rank of this bilinear map is~$3$.
In particular, the~decomposition~by~Gauss,~namely
\begin{equation}\label{eq:Gauss}
ac-bd = (a+b)\cdot c - b\cdot (c+d), ~~~~ ad+bc = (a+b)\cdot c + a\cdot (d-c),
\end{equation}
shows that the real rank is at most $3$ via the three 
multiplications $(a+b)\cdot c$, $b\cdot(c+d)$, and $a\cdot(d-c)$.

\subsubsection*{Over the complex numbers}
Let $T:\bC^2\times\bC^2\rightarrow\bC^2$ denote the complex multiplication bilinear map.
We first aim to compute $\brk T$ and~$\rk T$. 
Observe that $T\in \bC^2\otimes \bC^2\otimes \bC^2$ with its
rank and border rank corresponding to computing minimal decompositions with respect to the Segre variety $$S:=Seg(\bP^1\times \bP^1 \times \bP^1)\subset \bP(\bC^2\otimes \bC^2\otimes \bC^2)=\bP(\bC^8).$$
To accomplish this, we compute
a pseudowitness set for $S = \sigma_1(S)\subset\bP(\bC^8)$
for which $\deg \sigma_1(S) = 6$.
The membership test described in Prop.~\ref{prop:BasicMembership} tracked $6$ paths and found that each path 
converged to some finite endpoint which does not correspond to $T$. 
Therefore, $\rk T \geq \brk T > 1$. 

Next, we turn our attention to $\sigma_2(S)$,
which fills the whole space. Hence, we know $\brk T = 2$.
We use the method from Section~\ref{Sec:LargeRank} 
to track one solution path which indeed converges to a decomposition
of $T$ thereby showing $[T]\in\sigma_2^0(S)\subset\sigma_2(S)$,
i.e., $\rk T = \brk T = 2$.

For example, if we look for decompositions of the form
\begin{equation}\label{eq:DecompRestricted}
\begin{bmatrix}ac-bd\\ad+bc\end{bmatrix} = 
\begin{bmatrix}x_{11} & x_{12} \\x_{21} & x_{22}\end{bmatrix}
\begin{bmatrix}(a+y_{12} b)\cdot (c+z_{12}d) \\(y_{21}a+b)\cdot (z_{21}c+d)\end{bmatrix},
\end{equation}
then our setup tracking one path yielded the decomposition
$$
\begin{array}{rcl} ac - bd &=& ~\left((a-ib)\cdot(c-id) - (b-ia)\cdot(d-ic)\right)/2 \\
ad+bc &=& i\left((a-ib)\cdot(c-id) + (b-ia)\cdot(d-ic)\right)/2\end{array}$$
where $i = \sqrt{-1}$
with the two multiplication being 
$(a-ib)\cdot(c-id)$
and 
$(b-ia)\cdot(d-ic)$.

\subsubsection*{Over the real numbers}

Since \eqref{eq:Gauss} shows that $\rk_\bR T \leq 3$, we 
know that $\rk_\bR T = 3$ if and only if $\rk_\bR T > 2$.
In \eqref{eq:DecompRestricted}, we used a specialized
form to compute a decomposition over $\bC$.  
In this case, there were only finitely many decompositions
and all were not real. 

We could also work with a fully general formulation:
\begin{equation*}
\begin{bmatrix}ac-bd\\ad+bc\end{bmatrix} = 
\begin{bmatrix}x_{11} & x_{12} \\x_{21} & x_{22}\end{bmatrix}
\begin{bmatrix}(y_{11} a+y_{12} b)\cdot (z_{11}c+z_{12}d) \\(y_{21}a+y_{22}b)\cdot (z_{21}c+z_{22}d)\end{bmatrix}
\end{equation*}
which, by taking coefficients, defines a variety $V\subset\bC^{12}$, 
the union of two irreducible
varieties $V_1$ and $V_2$, each having dimension~$4$ and degree~$9$.
Using two different approaches, we show 
that $V\cap\bR^{12} = \emptyset$.

First, using the setup from Section~\ref{Sec:RealRank}, we compute the critical points
of the distance between~$V$ and a random point in $[-1,1]^{12}$.  
Since this yields $18$ critical points, all of which are nonreal, we know
$\rk_\bR T > 2$.

Alternatively, since the two irreducible components of $V$ are complex conjugates of each other, we know that $V\cap\bR^{12}$ is contained in $V_1\cap V_2$.  
Since $V_1\cap V_2 = \emptyset$, we again see that $\rk_\bR T > 2$.  

\subsection{A Coppersmith-Winograd tensor}\label{Ssec:CWtesnor}

In \cite{CW}, the following tensor in $\bC^3\otimes\bC^3\otimes\bC^3$ is considered:
$$T = a_1\otimes b_2 \otimes c_2
+ a_2\otimes b_1 \otimes c_2
+ a_2\otimes b_2 \otimes c_1
+ a_1\otimes b_3 \otimes c_3
+ a_3\otimes b_1 \otimes c_3
+ a_3\otimes b_3 \otimes c_1$$
where $\rk T = \brk T = 4$.  In fact, 
the secant variety $\sigma_4(\bC^3\times\bC^3\times\bC^3)$
is defective since it is expected to fill the ambient space 
but is actually a hypersurface of degree $9$ \cite{Stra83:laa}.

We used Prop.~\ref{prop:BasicMembership} upon computing 
a pseudowitness set for $\sigma_3(\bC^3\times\bC^3\times\bC^3)\subset\bC^{27}$, which has dimension~$21$ and degree $414$,
to verify that $\rk T \geq \brk T > 3$.  
In particular, using this pseudowitness set, 
the method of~\cite{aCM}
yields that $\sigma_3(\bC^3\times\bC^3\times\bC^3)$ is
arithmetically Cohen-Macaulay and defined by $27$ quartics.

We next compute all decompositions of $T$ of the form
\begin{equation}\label{eq:CWform}
\sum_{i=1}^4 (r_{i1} a_1 + r_{i2} a_2 + r_{i3} a_3)\otimes
(b_1 + s_{i2} b_2 + s_{i3} b_3)\otimes
(c_1 + t_{i2} c_2 + t_{i3} c_3).
\end{equation}
The tensor $T$ has a two-dimensional family 
of degree $30$ of decompositions of the form \eqref{eq:CWform}
which decomposes into $6$ irreducible components, each of degree $5$.  Hence, we have verified that $\rk T = \brk T = 4$.  
In fact, the~$6$ irreducible components arise
in three pairs of complex conjugates, say $V_i$ and $\overline{V_i}$ for $i = 1,2,3$.  Since, for each $i$, $V_i\cap\overline{V_i} = \emptyset$, $T$ does not have a real decomposition of the form \eqref{eq:CWform}.

\subsection{Comparison with cactus rank}\label{Ssec:Cactus}

The following example shows that our method computes $X$-border rank and not the \textit{cactus rank}, which was recently reintroduced in the literature 
(in \cite{IK}, it was defined as ``scheme length,'' and the first definition of cactus rank is in \cite{br} after paper \cite{bucbuc} where the cactus variety was first introduced). 
The following example was first published in \cite{bbm} thanks
to a suggestion from  W. Buczy\'nska and J. Buczy\'nski who proved it
in \cite{bucbuc} as a very peculiar but illustrative case where the
$X$-border rank of a polynomial cannot be computed from
a punctual scheme of the same length: 
$$T=x_0^2x_2+6x_1^2x_3-3(x_0+x_1)^2x_4.$$
The $X$-border rank of $T$ with respect to $X=\nu_3(\mathbb{P}^4)$ is $5$.
In fact, one can explicitly write down a family $T_\epsilon$ having rank $5$ with
$T_\epsilon\rightarrow T$, namely
$$
T_{\epsilon}= {1\over 3\epsilon  } (x_{0}+\epsilon x_{2})^{3} + 6(x_{1}+\epsilon x_{3})^{3}
-3(x_{0}+x_{1}+ \epsilon x_{4})^{3} + 3(x_{0} + 2\, x_{1})^{3} - (x_{0} + 3 x_{1})^{3}.
$$
However, it is not possible to find a scheme of length $5$ apolar to $T$
so that the cactus rank (and the rank) of $T$ is at least $6$
\cite{bbm,bucbuc}.

To verify that $\brk T > 4$, we compute a pseudowitness set for $\sigma_4(X)$, 
which was accomplished by starting with one point
and using $77$ random monodromy loops to generate additional points.
The trace test verified that $\deg \sigma_4(X)=\mbox{36,505}$. 
Thus, upon tracking 36,505 solution paths to perform
the membership test from Prop.~\ref{prop:BasicMembership}, 
we find that all converged and none of the endpoints correspond to $T$
yielding $\brk T > 4$.

A pseudowitness set for $\sigma_5(X)$ was computed
using a similar approach showing $\deg \sigma_5(X) =\mbox{24,047}$. 
After tracking 24,047 paths
to perform the membership test from Prop.~\ref{prop:BasicMembership}, we find a nonconvergent path whose projection converges to $T$ 
thereby showing $\brk T=5$ and providing an indication that $\rk T > 5$.

As with other examples, the pseudowitness sets computed for this example 
can be stored and reused to test whether other given cubic forms in $5$ variables
have border rank $4$ and $5$, respectively.

\subsection{Generic elements}\label{Ssec:GenDecomp}

We next compute decompositions of generic elements by
tracking one path as in Section~\ref{Sec:LargeRank}.

The following example was posed to one of the authors by M.~Mella a few years ago when the algorithm in \cite{OO} was not developed yet. 
In particular, Mella asked for a decomposition of a general polynomial of degree 5 in 3 variables, such as:
$$\mbox{\small
$\begin{array}{ll}
T~=&17051 x_0^5+41500 x_0^4 x_1+720 x_0^3 x_1^2+11360 x_0^2 x_1^3+95010 x_0 x_1^4~+19345 x_1^5-18095 x_0^4 x_2-281420 x_0^3 x_1 x_2\\
~&+~427290 x_0^2 x_1^2 x_2-367940 x_0 x_1^3 x_2+73860 x_1^4 x_2+243470 x_0^3 x_2^2-533370 x_0^2 x_1 x_2^2~+~518670 x_0 x_1^2 x_2^2\\
~&-~273140 x_1^3 x_2^2+156350 x_0^2 x_2^3-323300 x_0 x_1 x_2^3~+
383760 x_1^2 x_2^3+80245 x_0 x_2^4-277060 x_1 x_2^4+84411 x_2^5.
\end{array}$}$$

For $X = \nu_{5}(\bP^{2})$, $\sigma_7(X)$ fills the ambient 
space and we can compute a decomposition by tracking 
one solution path. 
Aiming to find a decomposition of the form
$$T=\sum_{j=1}^7 \lambda_j(x_0 + a_{j1} x_1 + a_{j2}x_2)^5,$$
the endpoint of our path yielded the decomposition 
$$
\begin{array}{ll}
T~= & 243(x_0 + 8/3 x_1 - 2/3 x_2)^5 - 
32768(x_0 - 3/4 x_1 + 1/8 x_2)^5~+ 16807(x_0 - x_1 + x_2)^5  \\
~&-~32(x_0 + 2 x_1 -4 x_2)^5~+ 32768(x_0 - 1/2 x_2)^5 + 
32(x_0 - 3/2 x_1 + 5/2 x_2)^5 + 
(x_0 -5 x_1 + 8 x_2)^5.
\end{array}$$

We note that the algorithm in \cite{OO} can decompose 
general tensors in 3 variables up to degree 6
whereas our numerical homotopy algorithm can be used to decompose polynomials of higher degree.  To illustrate, we start with a general element with a known decomposition, say 
$$\mbox{\small $\begin{array}{rl}
T~= &
91(x_0 -7/2 x_1 + 9/2x_2)^7 + 58(x_0 -3/2 x_1 - 4/3x_2)^7 - 21(x_0 + 2x_1 - 9/2x_2)^7
 +~33(x_0 + 3x_1 - x_2)^7 \\
 & ~+ 54(x_0 - 3x_1 - 5/3x_2)^7 + 88(x_0 - 3x_1 - 10/3x_2)^7 
 -~37(x_0 - 5x_1 + x_2)^7 + 93(x_0 - x_1 - 8x_2)^7 \\
 & ~+ 12(x_0 + 9/2x_1 + 10x_2)^7 
 -~89(x_0 - 5x_1 - 1/2x_2)^7 - 99(x_0 - x_1 - 3x_2)^7 - 22(x_0 - 1/3x_1 + 4x_2)^7.
\end{array}$}$$

After expanding $T$ to extract the coefficients,
which we rescale all of them to improve numerical conditioning,
we track one path in $36$ dimension.  
The resulting decomposition (where coefficients are rounded for readability and $i = \sqrt{-1}$) found is:
$$\mbox{\footnotesize $
\begin{array}{rl}
T~=&90.5217(x_0 - 1.0016x_1 - 8.0256x_2)^7 + 133.8171(x_0 - 3.6909x_1 - 2.8478x_2)^7 -~97.4074(x_0 - 5.0606x_1 + 0.2459x_2)^7 \\
& +~89.4516(x_0 - 3.4857x_1 + 4.5217x_2)^7 -~20.6552(x_0 - 0.3125x_1 + 4.125x_2)^7 + 12.0133(x_0 + 4.4986x_1 + 9.9992x_2)^7 \\
&+~32.5455(x_0 + 3.0145x_1 - x_2)^7 + 83.1754(x_0 - 2.3582x_1 - 1.5306x_2)^7 -~19.4167(x_0 + 2.0658x_1 - 4.4909x_2)^7 \\
&-~83.0069(x_0 - 4.3651x_1 - 2.1818x_2)^7 -~(30.0192+29.276i)(x_0 - (0.95833+0.2729i)x_1 - (3.9167+0.8299i)x_2)^7\\
&-~(30.0192-29.276i)(x_0 - (0.95833-0.2729i)x_1 - (3.9167-0.8299i)x_2)^7.
\end{array}$}$$

We note that the original decomposition and this one are simply two points
in the same fiber.  Starting from this computed decomposition, 
we can use the approach of \cite{HOOS} to compute all of the 
other points in the fiber.
In this case, we obtain four other decompositions, the 
one that we originally started with and the following three:
$$\mbox{\scriptsize $
\begin{array}{rl}
T~= &
-80.3535(x_0 - 0.96044x_1 - 3.2042x_2)^7 + 91.7624(x_0 - 3.4925x_1 + 4.4929x_2)^7 + 11.9529(x_0 + 4.5041x_1 + 10.005x_2)^7 \\
& -~42.331(x_0 - 5.1095x_1 - 1.1877x_2)^7 +~58.8757(x_0 - 1.9096x_1 - 0.70898x_2)^7 + 0.42442(x_0 - 6.5033x_1 - 3.8957x_2)^7 \\
& +~93.6035(x_0 - 1.0023x_1 - 7.9934x_2)^7 + 33.1366(x_0 + 2.9983x_1 - 1.0053x_2)^7 - 21.1233(x_0 + 2.0048x_1 - 4.4804x_2)^7 \\
&  -~81.3951(x_0 - 4.9804x_1 + 0.50977x_2)^7 +~121.1404(x_0 - 3.0446x_1 - 3.0528x_2)^7 + 12.4957(x_0 + 4.48x_1 + 9.9434x_2)^7 \\
= & -19.5517(x_0 - 0.49254x_1 + 4.36x_2)^7 - 1.4462(x_0 + 3.3148x_1 + 5.9615x_2)^7 - 24.6931(x_0 - 0.46377x_1 + 3.855x_2)^7\\
& -~64.3704(x_0 - 0.73438x_1 - 3.1471x_2)^7 +~94.5455(x_0 - 0.97674x_1 - 7.9818x_2)^7 - 18.506(x_0 + 1.7797x_1 - 4.9778x_2)^7 \\
&-~115.1045(x_0 - 5.0408x_1 + 0.0031746x_2)^7 + 30.4559(x_0 + 3.0345x_1 - 0.70492x_2)^7 \\
& +~126.7561(x_0 - 2.5128x_1 - 1.9074x_2)^7 +~(13.1591+9.58983i)(x_0 - (3.2017-1.1206i)x_1 - (4.0938-0.15711i)x_2)^7 \\
& +~89.4074(x_0 - 3.4483x_1 + 4.549x_2)^7 +~(13.1591-9.58983i)(x_0 - (3.2017+1.1206i)x_1 - (4.0938+0.15711i)x_2)^7 \\
= & -19.5946(x_0 - 0.21053x_1 + 4.1273x_2)^7 + 91.4966(x_0 + -0.99627x_1 + -8.0167x_2)^7  \\
&+~115.5185(x_0 - 2.7188x_1 - 3.4694x_2)^7  +~88.0263(x_0 - 3.5054x_1 + 4.5294x_2)^7 \\
&-~(3.5882+2.2523i)(x_0 - (5.4688+0.95833i)x_1 + (1.2571-0.93548i)x_2)^7 -~99.6415(x_0 - 5.1771x_1 - 0.12821x_2)^7\\
& -~(3.5882-2.2523i)(x_0 - (5.4688-0.95833i)x_1 + (1.2571+0.93548i)x_2)^7 -~23.76(x_0 + 1.9074x_1 - 4.4706x_2)^7 \\
&-~(14.6087-73.7949i)(x_0 - (1.6296-0.50355i)x_1 - (2.6154+1.0169i)x_2)^7 +~33.3191(x_0 + 2.9896x_1 - 0.99711x_2)^7\\
&-~(14.6087+73.7949i)(x_0 - (1.6296+0.50355i)x_1 - (2.6154-1.0169i)x_2)^7 + 12.0313(x_0 + 4.4976x_1 + 9.9967x_2)^7 \\
\end{array}$}$$
where $i = \sqrt{-1}$ and all numbers have been rounded for readability.

\subsection{A degree 110 hypersurface}\label{Ssec:Hypersurfaces}

In \cite{CTY10}, the authors consider the hypersurface $\mathcal{M}\subset\bP^{15}$ defined by the closure of elements of the form
$$p_{ijkl}=\left( \sum_{s=0}^1 a_{si}b_{sj}c_{sk}d_{sl} \right) \left( \sum_{r=0}^1 e_{ri}f_{rj}g_{rk}h_{rl} \right)\hbox{~for all~~}(i,j,k,l)\in\{0,1\}^4.$$
The variety $\mathcal{M}$ is a Hadamard product, namely $\mathcal{M}=\sigma_2(S). \sigma_2(S)$ where $S$ is the Segre embedding of $\mathbb{P}^1\times\mathbb{P}^1\times\mathbb{P}^1\times\mathbb{P}^1$ into $\mathbb{P}^{15}$.
The authors used this to show that $\deg \mathcal{M}=110$ and 
the Newton polytope of $\mathcal{M}$ has 17,214,912 
vertices.  However, they were unable to compute an explicit 
defining equation.  
Since our approach is based on witness and pseudowitness sets,
we do not need explicit equations to test for membership in $\mathcal{M}$.

Starting from one point on $\mathcal{M}$, we computed a pseudowitness set for $\mathcal{M}$ using $26$ monodromy loops.
This computation yielded $\deg \mathcal{M}=110$, as expected. 
Let $\mathcal{M}^0$ denote the corresponding constructible set so that $\mathcal{M}=\overline{\mathcal{M}^0}$. 
To demonstrate our membership test, we consider the point
$$v=(2,3,0,-1,4,2,0,1,1,-2,2,0,1,0,-4,3)\in\bP^{15}.$$
By tracking $110$ paths, we find that $v\not\in\mathcal{M}$. 
Next, we consider the point
$$\begin{array}{lcl}
w&=&(2528064,-3079104,-2340576,2038176,-1804032,2398464,1539648,-1524096,\\
~&~&1104000,456086,-2403720,284016,-511104,-502072,1220472,-23424)\in\bP^{15}.
\end{array}$$
In this case, our test yields $w\in\mathcal{M}^0$
with $w$ arising from
{\small
$$\begin{array}{llll}
a_{00}=-6.5220+1.8885i & c_{00}=-12.5203+0.5994i & e_{00}=-16.9364-9.3010i & g_{00}=-1.4597+9.8573i\\
a_{01}=17.1112+2.1887i & c_{01}=7.2261-2.6147i & e_{01}=-2.8396-0.2652i & g_{01}=-3.9468+8.2471i\\
a_{10}=-0.0000+0.0000i & c_{10}=-10.9459-1.5479i & e_{10}=-2.7150+6.8142i & g_{10}=-1.6724-1.2813i\\
a_{11}=4.7144+0.5813i & c_{11}=-17.8454+2.8591i & e_{11}=-0.1511-5.0503i & g_{11}=-2.9854-4.2021i\\
b_{00}=1.0901-1.3154i & d_{00}=18.1529+5.5948i & f_{00}=-3.0265-6.0562i & h_{00}=-0.3222-0.2848i\\
b_{01}=0.2466+0.3406i & d_{01}=11.1640-8.9931i & f_{01}=16.6817-14.5017i & h_{01}=-0.1314+1.0045i \\
b_{10}=-22.1726-13.8102i & d_{10}=0.0335-0.1311i & f_{10}=2.9081-0.6907i & h_{10}=-0.8416+7.6337i\\
b_{11}=0.7238-0.6901i & d_{11}=-0.0255-0.0286i & f_{11}=3.2373+6.1380i & h_{11}=3.0027-1.5781i\\
\end{array}$$
}
where $i=\sqrt{-1}$ and all decimals have been rounded for readability.

\subsection{Joins for decomposable polynomials}\label{Ssec:JoinDecomposable}

Consider the closure of the spaces in $\bP({\Sym}^4\bC^4)\subset\bP^{35}$
which can be written as the sum of $r$ squares of
quadrics and the sum of $s$ fourth powers of linear forms:
$$f = \sum_{i=1}^r q_i^2 + \sum_{j=1}^s \ell_j^4,$$
i.e., $J(\sigma_r(\nu_2(\mathbb{P}(\Sym^2\mathbb{C}^4))), \sigma_s(\nu_4(\mathbb{P}^3)))$.
The following lists the expected dimension, which
is the minimum of the dimension of the ambient space, namely $35$,
and $10r + 4s$, and the actual dimension for various $r$ and $s$.
The ones in bold correspond to the defective cases.
$$\begin{array}{|c||c|c|c|c|c|c|c|c|c|c||c|c|c|c|c|c|c|c|}
\hline
r & \multicolumn{10}{|c||}{0} & \multicolumn{8}{c|}{1}\\
\hline
s & 1 & 2 & 3 & 4 & 5 & 6 & 7 & 8 & 9 & 10& 0 & 1 & 2 & 3 &4&5&6&7\\
\hline
\hline
\hbox{Expected dim} &4&8&12&16&20&24&28&32&35&35& 10&14&18&22&26&30&34&35\\
\hline
\hbox{Actual dim} &4&8&12&16&20&24&28&32&{\bf 34}&35& 10&14&18&22&26&30&34&35 \\
\hline
\end{array}
$$
$$
\begin{array}{|c||c|c|c|c|c||c|c|c||c|c||c|}
\hline
r & \multicolumn{5}{|c||}{2} & \multicolumn{3}{c||}{3} & \multicolumn{2}{c||}{4} & 5\\
\hline
s & 0 & 1 & 2 & 3 & 4 & 0 & 1 & 2 & 0 & 1 & 0\\
\hline
\hline
\hbox{Expected dim} & 20 & 24 & 28 & 32 & 35 & 30 & 34 & 35 & 35 & 35 & 35 \\
\hline
\hbox{Actual dim} & {\bf 19} & {\bf 23} & {\bf 27} & {\bf 31} & 35 & {\bf 27} & {\bf 31} & 35 & {\bf 34} & 35 & 35\\
\hline
\end{array}
$$

We consider the two defective hypersurface cases.  
For the case $(r,s) = (0,9)$, we verified this yields a 
degree $10$ hypersurface \cite{AH1},
while the case $(r,s) = (4,0)$ is a degree~38,475 hypersurface \cite{BHORS}.

\subsection{Best low rank approximation}\label{Ssec:LowRank}

Motivated by \cite[Ex.~7 \& 8]{OSS:14}, we consider $S = \sigma_2(\nu_4(\bP^2))\subset\bP^{14}$.
Starting with one point in $S$, we computed a 
pseudowitness set for $S$ using $14$ monodromy
loops which yielded $\deg S = 75$.  
Thus, we can test membership in $S$ by tracking at most $75$ paths.  For example, consider the tensor from T. Schultz listed in \cite[Ex.~7]{OSS:14}:
$$ 
\begin{array}{rcl}
T &=& 0.1023 x_0^{4}
+0.0197 x_1^{4}
+0.1869 x_2^{4}
+0.0039 x_0^{2}x_1^{2}
+0.0407 x_0^{2}x_2^{2}
-0.00017418 x_1^{2} x_2^{2} \\
&&-~0.002 x_0^{3}x_1 
+0.0581 x_0^{3}x_2
+0.0107 x_0 x_1^{3}
+0.0196 x_0 x_2^{3}
+0.0029 x_1^{3}x_2 
 -0.0021 x_1x_2^{3}\\
&& -~0.00032569 x_0^{2}x_1x_2 
 -0.0012 x_0 x_1^{2} x_2
 -0.0011 x_0 x_1 x_2^{2}.
\end{array}
$$
Following the membership test from Section~\ref{Sec:Membership},
since all $75$ paths converged to points which did not correspond
to $[T]$, we know that $[T]\notin S$.  

Since it is expected that noise in the data moves 
an element off the variety, 
one often wants to compute the ``best'' low rank approximation.
In this case, we want a real element in $S$ which minimizes the 
Euclidean distance of the coefficients, treated as a vector in $\bC^{15}$.  
One approach is to compute all critical points which
was performed in \cite[Ex.~8]{OSS:14}.
This resulted in $195$ points outside of the set of rank $1$ elements, 
i.e., $\nu_4(\bP^2)$, of which $9$ are real.
In particular, there are $2$ are local minima and $7$ saddle points, with the
global minimum approximately being:
$$(0.0168 x_0 - 0.00189 x_1 + 0.657 x_2)^4 + 
(0.56 x_0 - 0.00254 x_1 + 0.0988 x_2)^4.$$

As discussed in Section~\ref{Sec:RealRank}, we can
also consider using local gradient descent homotopies 
to attempt to compute local minimizers of the distance function.  
In this case, since~$S$ is known to be defined by $148$ cubic polynomials, we utilized a random real combination of these polynomials.  In our experiment,
the path starting at $T$ produced a critical point 
of the distance function that was indeed the 
global minimizer above.  

\subsection{Skew-symmetric tensors}\label{Ssec:SkewSymm}

We next consider skew-symmetric tensors in $\bigwedge^3\mathbb{C}^7\subset\bC^{35}$ 
with respect to the Grassmannian $G(3,7)$.  By dimension counting, one
expects a general element to have rank $3$, but one can easily 
verify using \eqref{eq:DimensionImage} 
that $\sigma_3(G(3,7))$ is a hypersurface.
Hence, a general element has rank $4$.
The defectivity of this hypersurface has already been observed in
\cite{AOP,BDdG,CGG} and it is conjectured that, together with $\sigma_3(G(4,8))$, $\sigma_4(G(4,8))$, $\sigma_4(G(3,9))$ and their duals, they are the only 
defective secant varieties to a Grassmannian. 

To the best of our knowledge, the degree of this hypersurface
has not been computed before.  
By using a pseudowitness set computation, we find that 
this hypersurface has degree~$7$.  
Even though a degree $7$ polynomial defining this hypersurface is not known, we are able to decide membership in this hypersurface
by tracking at most $7$ paths.

We now turn to $\sigma_2(G(3,7))\subset\bC^{35}$ which is an
irreducible variety of dimension $26$ and degree~$735$.  In particular,
we aim to compute the codimension one components of its boundary as follows.
To simplify the computation, we consider the maps $A_i:\bC^5\rightarrow\bC^7$ defined by
$$\hbox{\scriptsize
$A_1(\alpha_1,\alpha_2,\alpha_3,\alpha_4,\alpha_5) = \left[\begin{array}{c} \alpha_1 \\ 0 \\ 0 \\ \alpha_2 \\ \alpha_3 \\ \alpha_4 \\ \alpha_5\end{array}\right],~~~
A_2(\alpha_1,\alpha_2,\alpha_3,\alpha_4,\alpha_5) = \left[\begin{array}{c} 0 \\ \alpha_1 \\ 0 \\ \alpha_2 \\ \alpha_3 \\ \alpha_4 \\ \alpha_5\end{array}\right],~~~
A_3(\alpha_1,\alpha_2,\alpha_3,\alpha_4,\alpha_5) = \left[\begin{array}{c} 0 \\ 0 \\ \alpha_1 \\ \alpha_2 \\ \alpha_3 \\ \alpha_4 \\ \alpha_5\end{array}\right].$}$$
Following Section~\ref{Sec:Boundary}, we slice
down to the curve case. 
For general affine linear polynomials $\ell_1,\dots,\ell_{25}$ in $35$ variables
and $p$ in $27$ variables, we consider the irreducible curve
{\scriptsize
$$\sC = \overline{\left.\left\{Z\cdot h^3 -
\sum_{i=1}^2 A_1(a^i_{1},a^i_{2},a^i_{3},a^i_{4},a^i_{5})\wedge A_2(a^i_{1},a^i_{6},a^i_{7},a^i_{8},a^i_{9})\wedge A_3(a^i_{1},a^i_{10},a^i_{11},a^i_{12},a^i_{13})~\right|~\ell_k(Z) = p(h,a^i_j)=0\right\}}.
$$}For a general $\beta\in\bC$, we found
that $\sC\cap\Var(h - \beta)$ consists of 48,930 points.  By tracking
the homotopy paths defined by $\sC\cap\Var(h - \beta\cdot t)$, 
44,520 paths yielded points in $\sC\cap\Var(h)$
with the corresponding points arising in two types.
The first type, which consists of 3262 distinct $Z$ coordinates,
each corresponding to the endpoint of $12$ paths, 
either have $a^1_1 = 0$ or $a^2_1 = 0$.  These points are
in the boundary based on the choice of parameterization used 
in $\sC$, but are not actually in the boundary of $\sigma_2(G(3,7))$.  
The second type, which consists of 1792 distinct $Z$ coordinates,
each corresponding to the endpoint of $3$ paths,
are indeed in the boundary of $\sigma_2(G(3,7))$.  
We note that $\mbox{44,520} = 12\cdot 3262 + 3\cdot 1792$.  
Moreover, the 1792 points form a witness point set for the following
irreducible variety of dimension $25$ and degree 1792:
$$\overline{\{v_1\wedge v_2 \wedge w_1 + v_1 \wedge w_2 \wedge v_3 + w_3 \wedge v_2 \wedge v_3~|~v_i,w_i\in\bC^7\}}\subset\bC^{35}$$
which is precisely the codimension one component of the boundary
of $\sigma_2(G(3,7))$.  

\medskip

In \cite{bv} the authors used the technique proposed by the present paper to compute the order of magnitude of the number of decompositions of a generic skew-symmetric tensor in $\bigwedge^4\mathbb{C}^9$ that is a perfect case.

\subsection{Matrix multiplication with zeros}\label{Ssec:MatrixMult}

We close with computing the border ranks of 
some tensors arising from the 
multiplication of a $2\times2$ matrix and a $2\times3$ matrix
with zero entries. 
One special case of this is the matrix multiplication tensor
of a $2\times2$ matrix with a zero term and a $2\times2$ matrix,
i.e., a $2\times3$ matrix with one column consisting of zeros.
In \cite{Bini5}, an explicit algorithm
shows that its border rank is $\leq 5$ 
with this observation leading
to an upper bound on the exponent of matrix 
multiplication $\omega$ of $\log_{12}1000 \approx 2.7799$.
Another reason for computing the border rank of 
such tensors arises from \cite{LR15} where the 
border rank of the matrix multiplication tensor 
for matrices of size $2\times 2$ and $2\times n$
is considered.  The results in \cite{LR15} 
build on computational results in~\cite{AS,Smirnov}.

The following table lists the border ranks
of various matrix multiplication tensors.  
In each matrix, an entry is $\ast$ if that entry 
can take an arbitrary value while $0$ means that entry is $0$. 

$$
\begin{array}{cccc}
\hbox{Number} & \hbox{Format} & \hbox{Ambient Space} & \hbox{Border Rank} \\
1 & \left[\begin{array}{cc} \ast & \ast \\ \ast & 0 \end{array}\right]
\cdot\left[\begin{array}{ccc} \ast & \ast & 0\\ \ast & \ast & 0 \end{array}\right] & \bC^3\otimes\bC^4\otimes\bC^4 & 5 \\ \vspace{-0.1in} \\
2 & \left[\begin{array}{cc} \ast & \ast \\ \ast & 0 \end{array}\right]
\cdot\left[\begin{array}{ccc} \ast & 0 & 0\\ \ast & \ast & \ast \end{array}\right] & \bC^3\otimes\bC^4\otimes\bC^4 & 5 \\\vspace{-0.1in} \\
3 & \left[\begin{array}{cc} \ast & \ast \\ \ast & 0 \end{array}\right]
\cdot\left[\begin{array}{ccc} \ast & 0 & \ast\\ \ast & \ast & 0 \end{array}\right] & \bC^3\otimes\bC^4\otimes\bC^5 & 6 \\\vspace{-0.1in} \\
4 & \left[\begin{array}{cc} \ast & \ast \\ \ast & 0 \end{array}\right]
\cdot\left[\begin{array}{ccc} \ast & \ast & \ast\\ \ast & 0 & 0 \end{array}\right] & \bC^3\otimes\bC^4\otimes\bC^6 & 6 \\\vspace{-0.1in} \\
5 & \left[\begin{array}{cc} \ast & \ast \\ \ast & 0 \end{array}\right]
\cdot\left[\begin{array}{ccc} \ast & \ast & 0\\ \ast & \ast & \ast \end{array}\right] & \bC^3\otimes\bC^5\otimes\bC^5 & 6 \\\vspace{-0.1in} \\
6 & \left[\begin{array}{cc} \ast & \ast \\ \ast & \ast \end{array}\right]
\cdot\left[\begin{array}{ccc} \ast & \ast & \ast \\ \ast & \ast & 0 \end{array}\right] & \bC^4\otimes\bC^5\otimes\bC^6 &  8\\\vspace{-0.1in} \\
\end{array}
$$

\paragraph{Number 1}
As mentioned above, the upper bound on border rank is provided in \cite{Bini5}.  The lower bound is provided
in \cite[Prop.~3.2]{LR15} which is based on
an equation provided in \cite{Stra83:laa}. 

\paragraph{Number 2}
The lower bound follows from 
using Prop.~\ref{prop:BasicMembership}
applied to the secant variety 
$\sigma_4(\bC^3\times\bC^4\times\bC^4)$
which has dimension $36$ and degree 252,776 \cite{D15}.
Alternatively, one could have followed a similar
approach as in \cite[Prop.~3.2]{LR15} 
using the defining equations for this secant 
variety, e.g., see \cite{BO}.
The upper bound is trivial since the 
standard definition of matrix multiplication
yields a rank $5$ decomposition.
Nonetheless, we note that the secant variety 
$\sigma_5(\bC^3\times\bC^4\times\bC^4)$
which has dimension $44$, i.e., it is defective, 
and degree $1716$.  In particular, the methods of \cite{aCM,aG}
show that $\sigma_5(\bC^3\times\bC^4\times\bC^4)$ is arithmetically Gorenstein and generated by 144 polynomials
of degree 11.  

\paragraph{Number 3}
The upper bound is trivial since the 
standard definition of matrix multiplication
yields a rank~$6$ decomposition.
Additionally, $\sigma_6(\bC^3\times\bC^4\times\bC^5)$ fills
the ambient space.
The lower bound is shown using Prop.~\ref{prop:BasicMembership}
applied to $\sigma_5(\bC^3\times\bC^4\times\bC^5)$
which has dimension $50$ and degree 581,584.

\paragraph{Number 4}
The upper bound is shown using Prop.~\ref{prop:BasicMembership}
applied to $\sigma_6(\bC^3\times\bC^4\times\bC^6)$
which has dimension 66 and degree 206,472.
To show the lower bound, consider the problem
of computing the $(1,1)$, $(1,2)$, $(2,1)$, and
$(2,2)$ entries and the sum of the $(1,3)$ and $(2,3)$
entries of the matrix product.  This is 
a problem in $\bC^3\otimes\bC^4\otimes\bC^5$
whose border rank is clearly a lower bound
on the border rank of the original problem.  
As in Number 3, $\sigma_5(\bC^3\times\bC^4\times\bC^5)$
with Prop.~\ref{prop:BasicMembership}
shows that a lower bound on the border rank is indeed $6$.

\paragraph{Number 5}
The lower bound follows immediately from Number 3
while the upper bound follows from \cite{Bini5} 
with one additional multiplication.

\paragraph{Number 6}
This was the motivating problem suggested
to the first and third authors by JM Landsberg
due to a gap between the upper bound of $8$ 
in \cite[Table~5]{Smirnov}
and the lower bound of $7$ from~\cite[Prop.~3.2]{LR15}.

Similar to Number 4, we will demonstrate a lower
bound by considering a subproblem.  In this case, 
we replace each entry of the $2\times2$ matrix with 
a random linear form in $3$ variables yielding
a problem in $\bC^3\otimes\bC^5\otimes\bC^6$.  
We showed the lower bound on the border rank of this 
subproblem is $8$ as follows.

Suppose that $a_1,a_2,a_3$, $b_1,\dots,b_5$, and $c_1,\dots,c_6$
are the standard basis elements for $\bC^3$, $\bC^5$, and $\bC^6$, respectively.  Consider
the projection $\pi:\bC^{90}\rightarrow\bC^{89}$ that ignores the entry $a_3\otimes b_5\otimes c_6\in\bC^3\otimes\bC^5\otimes\bC^6 \simeq \bC^{90}$.  Then, $\overline{\pi(\sigma_7(\bC^3\times\bC^5\times\bC^6))}$ is a variety in $\bC^{89}$ of dimension 
$84$ and degree 455,176 for which the membership test
shows that it does not contain the image under $\pi$ of the subproblem tensor.


\pdfbookmark[1]{Acknowledgements}{Sec:Ack}
\section*{Acknowledgements}\label{Sec:Ack}

AB and JDH would like to thank
the Institut Mittag-Leffler (Djursholm, Sweden)
for their support and hospitality
which is where many of the ideas of this paper were first conceived.

\pdfbookmark[1]{References}{thebibliography}


\begin{thebibliography}{99}

{\small

\bibitem{AOP} H. Abo, G. Ottaviani and C. Peterson.
\newblock Non-defectivity of Grassmannians of planes.
\newblock {\em J. Algebr. Geom.} 21(1), 1--20, 2012.

\bibitem{AMS}R. Achilles, M. Manaresi and P.  Schenzel.
\newblock A degree formula for secant varieties of curves.
\newblock {\em Proceedings of the Edinburgh Mathematical Society}, 57, 305--322, 2014.

\bibitem{AS} V.B. Alekseev and A.V. Smirnov.
\newblock {On the exact and approximate bilinear complexities of
              multiplication of {$4\times 2$} and {$2\times 2$} matrices}.
\newblock {\em Proc. Steklov Inst. Math.}, 282, 123--139, 2013.

\bibitem{AH1} J. Alexander and A. Hirschowitz.
\newblock Polynomial interpolation in several variables.
\newblock {\em J. Algebr. Geom.}, 4(2), 201--222, 1995.

\bibitem{ar} E.S. Allman and J.A. Rhodes.
\newblock Phylogenetic ideals and varieties for the general Markov model.
 \newblock {\em Adv. in Appl. Math.}, 40(2), 127--148, 2008.
\bibitem{RealSolving} P. Aubry, F. Rouillier, and M. Safey El Din.
\newblock Real solving for positive dimensional systems.
\newblock {\em J. Symbolic Comput.}, 34 (6), 543--560, 2002.

\bibitem{Ba} E. Ballico. 
\newblock On the typical rank of real polynomials (or symmetric tensors) with a fixed border rank.
\newblock {\em Acta Mathematica Vietnamica}, 39(3), 367--378, 2014.

\bibitem{babe} E. Ballico and A. Bernardi.
\newblock Stratification of the fourth secant variety of Veronese variety via the symmetric rank.
\newblock {\em Adv. Pure Appl. Math.}, 4(2), 215--250, 2013.

\bibitem{babe2} E. Ballico, A. Bernardi.
\newblock Tensor ranks on tangent developable of Segre varieties.  
\newblock {\em Linear Multilinear Algebra}, 61(7), 881--984, 2013. 

\bibitem{Ban} M. Banchi.
\newblock Rank and border rank of real ternary cubics.
\newblock {\em Boll. Unione Mat. Ital.}, 8(1), 65--80, 2015.

\bibitem{Exactness} D.J. Bates, J.D. Hauenstein, T.M. McCoy, C. Peterson, and A.J. Sommese.
\newblock Recovering exact results from inexact numerical data in algebraic geometry.
\newblock {\em Exp. Math.}, 22(1), 38--50, 2013.

\bibitem{Bertini} D.J. Bates, J.D. Hauenstein, A.J. Sommese, and C.W. Wampler.
\newblock {\em Bertini: software for numerical algebraic geometry}.
\newblock Available at \url{bertini.nd.edu}.

\bibitem{BertiniBook} D.J. Bates, J.D. Hauenstein, A.J. Sommese, and C.W. Wampler.
\newblock {\em Numerically Solving Polynomial Systems with Bertini}.
\newblock Volume 25 of Software, Environments, and Tools, SIAM, Philadelphia, 2013.

\bibitem{BO} D.J. Bates and L. Oeding.
\newblock Toward a Salmon conjecture.
\newblock {\em Exp. Math.}, 20(3), 358--370, 2011.

\bibitem{BDdG} K. Baur, J. Draisma and W. de Graaf.
\newblock Secant dimensions of minimal orbits: computations and conjectures.
\newblock {\em Exp. Math.} 16(2), 239--250, 2007.

\bibitem{bgi} A. Bernardi, A. Gimigliano and M. Id\`a.
\newblock Computing symmetric rank for symmetric tensors.
\newblock {\em J. Symbolic Comput.}, 46, 34--53, 2011.

\bibitem{bbcm} A. Bernardi, J. Brachat, P. Comon and B. Mourrain.
\newblock General tensor decomposition, moment matrices and applications.  
\newblock {\em J. Symbolic Comput.} 52, 51--71, 2013 

\bibitem{bbm} A. Bernardi, J. Brachat and B. Mourrain.
\newblock A comparison of different notions of ranks of symmetric tensors.
\newblock {\em LAA}, 460,205--230, 2014

\bibitem{bbo} A. Bernardi, G. Blekherman and G. Ottaviani.
\newblock On real typical ranks.
\newblock Preprint, {\tt arXiv:1512.01853}, 2015.

\bibitem{br} A. Bernardi and K. Ranestad.
\newblock On the cactus rank of cubic forms. 
\newblock {\em J. Symbolic Comput.} 50, 291--297, 2013.

\bibitem{bv} A. Bernardi and D. Vanzo.
\newblock A new class of non-identifiable skew symmetric tensors.
\newblock Preprint, {\tt arXiv:1606.04158}, 2016.

\bibitem{Bini5} D. Bini, M. Capovani, F. Romani, 
and G. Lotti.
\newblock $O(n^{2.7799})$ complexity for $n\times n$ approximate matrix multiplication,
\newblock {\em Inform. Process. Lett.}, 8(5), 234--235, 1979.

\bibitem{Ble} G. Blekherman.
\newblock Typical real ranks of binary forms.
\newblock {\em Found. Comput. Math.}, 15(3), 793--798, 2015.

\bibitem{BHORS} G. Blekherman, J.D. Hauenstein, J.C. Ottem, K. Ranestand, and B. Sturmfels.
\newblock Algebraic boundaries of Hilbert's SOS cones.
\newblock {\em Compositio Mathematica}, 148(6), 1717--1735, 2012. 

\bibitem{BZ} G. Blekherman and Z. Teitler.
\newblock On maximum, typical, and generic ranks.
\newblock {\em Math. Ann.}, 362(3), 1021--1031, 2015.

\bibitem{bcg} M. Boji, E. Carlini and A.V. Geramita.
\newblock Monomials as sums of powers: the Real binary case.
\newblock {\em Proc. Amer. Math. Soc.}, 139, 3039--3043, 2011.

\bibitem{BCMT} J. Brachat, P. Comon, B. Mourrain and E.P. Tsigaridas.
\newblock Symmetric tensor decomposition.
\newblock {\em Linear Algebra and Applications}, Elsevier - Academic Press,  433(11--12), 1851--1872, 2010.

\bibitem{bucbuc} W. Buczy\'nska and J. Buczy\'nski.
\newblock On differences between the border rank and the smoothable rank of a polynomial.
\newblock {\em Glasgow Math. J.}, 57, 401--413, 2015.

\bibitem{ccg} E. Carlini, M.V. Catalisano and A.V. Geramita.
\newblock The solution to the Waring problem for monomials and the sum of coprime monomials.
\newblock {\em J. of Algebra}, 370,  5--14, 2012.

\bibitem{CGG} M.V. Catalisano, A.V. Geramita and A. Gimigliano.
\newblock  Secant varieties of Grassmann varieties.
\newblock {\em Proc. Amer. Math. Soc.} 133(3), 633--642, 2005.

\bibitem{CRe} A. Causa and R. Re.
\newblock{On the maximum rank of a real binary form}. 
\newblock {\em Annali di Matematica Pura ed Applicata}, 190(1), 55--59, 2011.

\bibitem{ACCF} P. Chevalier, L. Albera, A. Ferreol, and P. Comon.
\newblock On the virtual array concept for higher order array processing.
\newblock {\em IEEE Trans. Sig. Proc.}, 53, 1254--1271, 2005.

\bibitem{CS} G. Comas and M. Seiguer.
\newblock On the rank of a binary form.
\newblock {\em Found. Comput. Math.}, 11(1), 65--78, 2011.

\bibitem{Comon} P. Comon.
\newblock Independent component analysis,
\newblock {\em Higher-Order Statistics}, J.L.Lacoume ed., Elsevier,  29--38, 1992.

\bibitem{Co2} P. Comon.
\newblock Tensor decompositions.
\newblock In {\em Math. Signal Processing V}, J.G. Mc Whirter and I.K. Proudler eds., Clarendon press, Oxford, UK,
\newblock 1--24, 2002.

\bibitem{CoJu} P. Comon and C. Jutten.
\newblock Handbook of Blind Source Separation: Independent Component Analysis and Applications.
\newblock {\em Academic Press}. London, UK, 2010.

\bibitem{CO} P. Comon and G. Ottaviani.
\newblock On the typical rank of real binary forms.
\newblock {\em Linear Multilinear Algebra}, 60(6), 657--667, 2012.

\bibitem{CW} D. Coppersmith and S. Winograd.
\newblock Matrix multiplication via arithmetic progressions.
\newblock {\em J. Symbolic Comput.}, 9(3), 251--280, 1990.

\bibitem{cs} D. Cox and J. Sidman.
\newblock Secant varieties of toric varieties.
\newblock {\em J. Pure Appl. Algebr.}, 209(3), 651--669, 2007.

\bibitem{CTY10} M.A. Cueto, E.A. Tobis, and J. Yu.
\newblock An implicitization challenge for binary factor analysis.
\newblock {\em J. Symbolic Comput.}, 45(12), 1296--1315, 2010.

\bibitem{D15} N.S. Daleo.
\newblock Algorithms and Applications in Numerical Elimination Theory.
\newblock PhD dissertation. North Carolina State University, 2015.

\bibitem{aCM} N.S. Daleo and J.D. Hauenstein.
\newblock Numerically deciding the arithmetically
Cohen-Macaulayness of a projective scheme.
\newblock {\em J. Symbolic Comput.}, 72, 128--146, 2016.

\bibitem{aG} N.S. Daleo and J.D. Hauenstein.
\newblock Numerically testing generically reduced projective schemes for the arithmetic Gorenstein property.
\newblock {\em LNCS}, 9582, 137--142, 2016.

\bibitem{DHO14} N.S. Daleo, J.D. Hauenstein, and L. Oeding.
\newblock Computations and equations for Segre-Grassmann hypersurfaces.
\newblock {\em Port. Math.}, 73(1), 71--90, 2016.

\bibitem{deLC} L. de Lathauwer and J. Castaing.
\newblock Tensor-Based Techniques for the Blind Separation of {DS-CDMA} Signals.
\newblock {\em Signal Processing}.
87(2), 322--336, 2007.

\bibitem{EuclideanDistance} J. Draisma, E. Horobe\c{t}, G. Ottaviani, B. Sturmfels, and R. Thomas.
\newblock The {E}uclidean distance degree of an algebraic variety.
\newblock {\em Found. Comput. Math.}, 16(1), 99--149, 2016.

\bibitem{entanglement} J. Eisert and D. Gross.
\newblock Multiparticle entanglement.
\newblock In {\em Bru\ss, Dagmar (ed.) et al., Lectures on quantum information. Weinheim: Wiley-VCH. Physics Textbook}, 237--252, 2007.

\bibitem{ES} G. Ellingsrud and S. A. Stromme
\newblock Bott's formula and enumerative geometry.
\newblock {\em J. Amer. Math. Soc.} 9 (1996), 175--193.

\bibitem{GH13} Z.A. Griffin and J.D. Hauenstein.
\newblock Real solutions to systems of polynomial equations and parameter continuation.
\newblock {\em Adv. Geom.}, 15(2), 173--187, 2015.

\bibitem{H13} J.D. Hauenstein.
\newblock Numerically computing real points on algebraic sets.
\newblock {\em App. Math.}, 125(1), 105--119, 2013.

\bibitem{HHL14} J.D. Hauenstein, I. Haywood, and A.C. Liddell, Jr.
\newblock An a posteriori certification algorithm for {N}ewton homotopies.
\newblock In {\em ISSAC 2014}, ACM, New York, pp. 248--255. 

\bibitem{HKL13} J.D. Hauenstein, C. Ikenmeyer, and J.M. Landsberg.
\newblock Equations for lower bounds on border rank.
\newblock {\em Exp. Math.}, 22(4), 372--383, 2013. 

\bibitem{HL14} J.D. Hauenstein and A.C. Liddell, Jr.
\newblock Certified predictor-corrector tracking for {N}ewton homotopies.
\newblock {\em J. Symbolic Comput.}, 74, 239--254, 2016.

\bibitem{HMS15} J.D. Hauenstein, B. Mourrain, and A. Szanto.
\newblock Certifying isolated singular points and their multiplicity structure.
\newblock To appear in {\em  J. Symbolic Comput.}.

\bibitem{HOOS} J.D. Hauenstein, L. Oeding, G. Ottaviani, and A.J. Sommese.
\newblock Homotopy techniques for tensor decomposition and perfect identifiability.
\newblock Preprint, {\tt arXiv:1501.00090}, 2015.

\bibitem{HS13} J.D. Hauenstein and A.J. Sommese.
\newblock Membership tests for images of algebraic sets by linear projections.
\newblock {\em Appl. Math. Comput.}, 219(12), 6809--6818, 2013.

\bibitem{HS10} J.D. Hauenstein and A.J. Sommese.
\newblock Witness sets of projections.
\newblock {\em Appl. Math. Comput.}, 217(7), 3349--3354, 2010.

\bibitem{RegenCascade} J.D. Hauenstein, A.J. Sommese, and C.W. Wampler.
\newblock Regenerative cascade homotopies for solving polynomial systems.
\newblock {\em Appl. Math. Comput.} 218(4), 1240--1246, 2011.

\bibitem{HW13} J.D. Hauenstein and C.W. Wampler.
\newblock Numerical algebraic intersection using regeneration.
\newblock Preprint available at \url{www.nd.edu/~jhauenst/preprints}.

\bibitem{Iso} J.D. Hauenstein and C.W. Wampler.
\newblock Isosingular sets and deflation.
\newblock {\em Found. Comput. Math.}, 13(3), 371--403, 2013. 

\bibitem{IK} A.A.  Iarrobino and V. Kanev.
\newblock Power sums, Gorenstein algebras, and determinantal loci. {\em Lecture Notes in Mathematics}, 1721, Springer-Verlag, Berlin, Appendix C by Iarrobino and Steven L. Kleiman. 1999.

\bibitem{JS} T. Jiang and  N.D. Sidiropoulos.
\newblock Kruskal's permutation lemma and the identification of CANDECOMP\-/PARAFAC and bilinear models with constant modulus constraints.
\newblock {\em IEEE Trans. Sig. Proc.}, 52(9), 2625--2636, 2004.

\bibitem{Kanev} V. Kanev.
\newblock Chordal varieties of Veronese varieties and catalecticant matrices.
\newblock {\em J. Math. Sci.}, 94 (1), 1114--1125, 1999.

\bibitem{KBSIAMReview} T.G. Kolda and B.W. Bader.
\newblock Tensor Decompositions and Applications.
\newblock {\em SIAM Review}, 51(3), 455--500, 2009.

\bibitem{JML} J.M. Landsberg.
\newblock Tensors: geometry and applications.
\newblock Graduate Studies in Mathematics, vol. 128, {\em American Mathematical Society}, Providence, RI, 2012.

\bibitem{LO} J.M. Landsberg and G. Ottaviani.
\newblock Equations for secant varieties of Veronese and other varieties.
\newblock {\em Annali di Matematica Pura e Applicata}, 192, 569--606, 2013.

\bibitem{LR15} J.M. Landsberg and N. Ryder.
\newblock On the geometry of border rank algorithms
for $n\times 2$ and $2\times 2$ matrix multiplication.
\newblock Preprint, {\tt arXiv:1509.08323}, 2015.

\bibitem{LT} J.M. Landsberg and Z. Teitler.
\newblock On the ranks and border ranks of symmetric tensors.
\newblock {\em Found. Comput. Math.}, 10(3), 339--366, 2010.

\bibitem{LSClaw} C. Long and S. Sullivant.
\newblock Tying up loose strands: defining equations of the strand               symmetric model.
\newblock {\em J. Algebr. Stat.}, 6(1), 17--23, 2015.

\bibitem{MR15} A. Martin del Campo and J.I. Rodriguez.
\newblock Critical points via monodromy and local methods.
\newblock To appear in {\em J. Symbolic Comput}.

\bibitem{McC} P. McCullagh.
\newblock Tensor Methods in Statistics.
\newblock {\em Monographs on Statistics and Applied Probability}, Chapman \& Hall, London, 1987.

\bibitem{CoeffParam} A.P. Morgan and A.J. Sommese.
\newblock Coefficient-parameter polynomial continuation.
\newblock {\em Appl. Math. Comput.}, 29 (2), 123--160, 1989.
\newblock Errata: {\em Appl. Math. Comput.}, 51, 207, 1992.

\bibitem{OO} L. Oeding and G. Ottaviani.
\newblock Eigenvectors of tensors and algorithms for Waring decomposition.
\newblock {\em J. Symbolic Comput.}, 54, 9--35, 2013.

\bibitem{OSS:14} G. Ottaviani, P.-J. Spaenlehauer  and B. Sturmfels,
\newblock Exact solutions in structured low-rank approximation.
\newblock {\em SIAM J. Matrix Anal. Appl.}, 35(4), 1521-01542, 2014.

\bibitem{Raicu} C. Raicu.
\newblock Secant varieties of Segre--Veronese varieties.
\newblock {\em Algebra \& Number Theory}, 8, 1817--1868, 2012.

\bibitem{ran} K. Ranestad.
\newblock The degree of the secant variety and the join of monomial curves.
\newblock {\em Collect. Math.} 57, 1 (2006), 27--41.

\bibitem{rs} K. Ranestad and F-O. Schreyer.
\newblock On the rank of a symmetric form.
\newblock {\em J. Algebra},  346(1), 340--342, 2011.

\bibitem{Real} F. Rouillier, M.-F. Roy, and M. Safey El Din.
\newblock Finding at least one point in each connected component of a real algebraic set defined by a single equation.
\newblock {\em J. Complexity}, 16 (4), 716--750, 2000.

\bibitem{Schreyer} F.-O. Schreyer.
\newblock Geometry and algebra of prime Fano 3-folds of genus 12.
\newblock{\em Compositio Math.}, 127(3), 297--319, 2001.

\bibitem{ScSe} T. Schultz and H.P. Seidel.
\newblock Estimating crossing fibers: a tensor decomposition approach.
\newblock {\em IEEE Trans Vis Comput Graph},
48(6), 1635--42, 2008.

\bibitem{Seidenberg} A. Seidenberg.
\newblock A new decision method for elementary algebra.
\newblock {\em Ann. of Math. (2)}, 60, 365--374, 1954.

\bibitem{SGB} N.D. Sidiropoulos, G.B. Giannakis and R. Bro.
\newblock Blind {PARAFAC} Receivers for {DS-CDMA} Systems.
\newblock {\em IEEE Trans. on Sig. Proc.},
48(3), 810--823, 2000.

\bibitem{SBG} A. Smilde, R. Bro and P. Geladi.
\newblock Multi-Way Analysis,
\newblock {\em Wiley}, 2004.

\bibitem{Smirnov} A.V. Smirnov.
\newblock The bilinear complexity and practical algorithms for matrix
              multiplication.
\newblock {\em Zh. Vychisl. Mat. Mat. Fiz.}, 53(12), 1970--1984, 2013.

\bibitem{Cascade} A.J. Sommese and J. Verschelde.
\newblock Numerical homotopies to compute generic points on positive dimensional algebraic sets.
\newblock {\em J. Complexity}, 16(3), 572--602, 2000.

\bibitem{SVW04} A.J. Sommese, J. Verschelde, and C.W. Wampler.
\newblock Homotopies for intersecting solution components of polynomial systems.
\newblock {\em SIAM J. Numer. Anal.}, 42(4), 1552--1571, 2004.

\bibitem{SVW01} A.J. Sommese, J. Verschelde, and C.W. Wampler.
\newblock Numerical irreducible decomposition using projections from points on components.
\newblock {\em Contemp. Math.}, 206, 37--51, 2001.

\bibitem{SVW03} A.J. Sommese, Jan Verschelde, and C.W. Wampler.
\newblock Numerical irreducible decomposition using PHCpack.
\newblock In {\em Mathematics and Visualization}, ed. M. Joswig and N. Takayama, Springer-Verlag, 2003, pp. 109--130. 

\bibitem{SW05} A.J. Sommese and C.W. Wampler.
\newblock {\em The Numerical solution of systems of
          polynomials arising in engineering and science.}
\newblock World Scientific Press, Singapore, 2005.

\bibitem{tensorlab} L. Sorber, M. Van Barel and L. De Lathauwer. 
\newblock Tensorlab v2.0.
\newblock Available at \url{www.tensorlab.net}, 2014.

\bibitem{Stra83:laa} V. Strassen.
\newblock Rank and optimal computation of generic tensors.
\newblock {\em Linear Algebra Appl.}, 52, 645--685, 1983.

\bibitem{Syl} J.J. Sylvester.
\newblock Sur une extension d'un th\'eor\`eme de Clebsh relatif aux courbes du quatri\`eme degr\'e.
\newblock {\em Comptes Rendus, Math. Acad. Sci. Paris}, 102, 1532--1534, 1886.

\bibitem{Vali} L.G. Valiant.
\newblock Quantum computers that can be simulated classically in polynomial time.
\newblock {\em Proceedings of the {T}hirty-{T}hird {A}nnual {ACM} {S}ymposium on {T}heory of {C}omputing},
 {114--123 (electronic)},
 {ACM},
   {New York},
{2001}.

\bibitem{WuReid} W. Wu and G. Reid.
\newblock Finding points on real solution components and applications to differential polynomial systems. 
\newblock In {\em ISSAC 2013}, ACM, New York, 2013, pp. 339--346.
}
\end{thebibliography}
\end{document}